\documentclass[11pt, a4paper]{amsart}
\pdfoutput=1 
\usepackage[dvips]{graphicx}
\usepackage{amssymb, amstext, amscd, amsmath, color}
\usepackage{epstopdf}
\usepackage{tikz,enumerate}
\usepackage{url}
\usepackage{tikz-cd}
\usetikzlibrary{arrows}

\definecolor{cof}{RGB}{219,144,71}
\definecolor{pur}{RGB}{186,146,162}
\definecolor{greeo}{RGB}{91,173,69}
\definecolor{greet}{RGB}{52,111,72}

\usepackage{hhline}
\usepackage{hyperref}

\usepackage{kbordermatrix} 
\renewcommand{\kbldelim}{[} 
\renewcommand{\kbrdelim}{]}

\usepackage{etoolbox}
\let\bbordermatrix\bordermatrix
\patchcmd{\bbordermatrix}{8.75}{4.75}{}{}
\patchcmd{\bbordermatrix}{\left(}{\left[}{}{}
\patchcmd{\bbordermatrix}{\right)}{\right]}{}{}

\tikzstyle{vertex}=[circle, draw, fill=black, inner sep=0pt, minimum size=4pt]
\tikzstyle{blankvertex}=[circle, draw=white, fill=white, inner sep=0pt, minimum size=4pt]
\tikzstyle{edge}=[line width=1pt]
\tikzstyle{dashedge}=[line width=1.5pt, dashed]

\begin{document}

\newtheorem{theorem}{Theorem}[section]
\newtheorem{corollary}[theorem]{Corollary}
\newtheorem{proposition}[theorem]{Proposition}
\newtheorem{lemma}[theorem]{Lemma}

\theoremstyle{definition}
\newtheorem{remark}[theorem]{Remark}
\newtheorem{definition}[theorem]{Definition}
\newtheorem{example}[theorem]{Example}
\newtheorem{conjecture}[theorem]{Conjecture}

\newcommand{\FFock}{\mathcal{F}}
\newcommand{\kil}{\mathsf{k}}
\newcommand{\Hil}{\mathsf{H}}
\newcommand{\hil}{\mathsf{h}}
\newcommand{\Kil}{\mathsf{K}}
\newcommand{\Real}{\mathbb{R}}
\newcommand{\Rplus}{\Real_+}

\newcommand{\bC}{{\mathbb{C}}}
\newcommand{\bD}{{\mathbb{D}}}
\newcommand{\bN}{{\mathbb{N}}}
\newcommand{\bQ}{{\mathbb{Q}}}
\newcommand{\bR}{{\mathbb{R}}}
\newcommand{\bT}{{\mathbb{T}}}
\newcommand{\bX}{{\mathbb{X}}}
\newcommand{\bZ}{{\mathbb{Z}}}
\newcommand{\bH}{{\mathbb{H}}}
\newcommand{\BH}{{\B(\H)}}
\newcommand{\bsl}{\setminus}
\newcommand{\ca}{\mathrm{C}^*}
\newcommand{\cstar}{\mathrm{C}^*}
\newcommand{\cenv}{\mathrm{C}^*_{\text{env}}}
\newcommand{\rip}{\rangle}
\newcommand{\ol}{\overline}
\newcommand{\td}{\widetilde}
\newcommand{\wh}{\widehat}
\newcommand{\sot}{\textsc{sot}}
\newcommand{\wot}{\textsc{wot}}
\newcommand{\wotclos}[1]{\ol{#1}^{\textsc{wot}}}
 \newcommand{\A}{{\mathcal{A}}}
 \newcommand{\B}{{\mathcal{B}}}
 \newcommand{\C}{{\mathcal{C}}}
 \newcommand{\D}{{\mathcal{D}}}
 \newcommand{\E}{{\mathcal{E}}}
 \newcommand{\F}{{\mathcal{F}}}
 \newcommand{\G}{{\mathcal{G}}}
\renewcommand{\H}{{\mathcal{H}}}
 \newcommand{\I}{{\mathcal{I}}}
 \newcommand{\J}{{\mathcal{J}}}
 \newcommand{\K}{{\mathcal{K}}}
\renewcommand{\L}{{\mathcal{L}}}
 \newcommand{\M}{{\mathcal{M}}}
 \newcommand{\N}{{\mathcal{N}}}
\renewcommand{\O}{{\mathcal{O}}}
\renewcommand{\P}{{\mathcal{P}}}
 \newcommand{\Q}{{\mathcal{Q}}}
 \newcommand{\R}{{\mathcal{R}}}
\renewcommand{\S}{{\mathcal{S}}}
 \newcommand{\T}{{\mathcal{T}}}
 \newcommand{\U}{{\mathcal{U}}}
 \newcommand{\V}{{\mathcal{V}}}
 \newcommand{\W}{{\mathcal{W}}}
 \newcommand{\X}{{\mathcal{X}}}
 \newcommand{\Y}{{\mathcal{Y}}}
 \newcommand{\Z}{{\mathcal{Z}}}

\newcommand{\supp}{\operatorname{supp}}
\newcommand{\conv}{\operatorname{conv}}
\newcommand{\cone}{\operatorname{cone}}
\newcommand{\vspan}{\operatorname{span}}
\newcommand{\proj}{\operatorname{proj}}
\newcommand{\sgn}{\operatorname{sgn}}
\newcommand{\rank}{\operatorname{rank}}
\newcommand{\Isom}{\operatorname{Isom}}
\newcommand{\qIsom}{\operatorname{q-Isom}}
\newcommand{\Cknet}{{\mathcal{C}_{\text{knet}}}}
\newcommand{\Ckag}{{\mathcal{C}_{\text{kag}}}}
\newcommand{\rind}{\operatorname{r-ind}}
\newcommand{\lind}{\operatorname{r-ind}}
\newcommand{\ind}{\operatorname{ind}}
\newcommand{\coker}{\operatorname{coker}}
\newcommand{\Aut}{\operatorname{Aut}}
\newcommand{\Hom}{\operatorname{Hom}}
\newcommand{\GL}{\operatorname{GL}}
\newcommand{\tr}{\operatorname{tr}}
\newcommand{\Reg}{\operatorname{Reg}}
\newcommand{\End}{\operatorname{end}}

\newcommand{\eqnwithbr}[2]{%
\refstepcounter{equation}
\begin{trivlist}
\item[]#1 \hfill $\displaystyle #2$ \hfill (\theequation)
\end{trivlist}}

\setcounter{tocdepth}{1}

\title[Which graphs are rigid in $\ell_p^d$?]{Which graphs are rigid in $\ell_p^d$?}

\author[S.~Dewar, D.~Kitson and A.~Nixon]{Sean Dewar, Derek Kitson and Anthony Nixon}

\address{Johann Radon Institute of Computational and Applied Mathematics (RICAM)
\\ Austrian Academy of Sciences \\ Linz \\Austria.}
\email{sean.dewar@ricam.oeaw.ac.at}

\address{Dept.\ Math.\ Stats.\\ Lancaster University\\
Lancaster, LA1 4YF \\U.K.}
\address{Dept.\ Math.\ Comp. St.\\Mary Immaculate College, Thurles, Co.~Tipperary, Ireland.}
\email{derek.kitson@mic.ul.ie}

\address{Dept.\ Math.\ Stats.\\ Lancaster University\\
Lancaster LA1 4YF \\U.K.}
\email{a.nixon@lancaster.ac.uk}

\thanks{D.K. supported by the Engineering and Physical Sciences Research Council [grant numbers EP/P01108X/1 and EP/S00940X/1].
S.D. supported by the Austrian Science Fund (FWF): P31888.}

\subjclass[2010]{52C25, 05C50}
\keywords{bar-joint framework, infinitesimal rigidity, rigidity matroid, normed spaces}

\begin{abstract}
We present three results which support the conjecture that a graph is minimally rigid in $d$-dimensional $\ell_p$-space, where $p\in (1,\infty)$ and $p\not=2$, if and only if it is $(d,d)$-tight. Firstly, we introduce a graph bracing operation which preserves independence in the generic rigidity matroid when passing from $\ell_p^d$ to $\ell_p^{d+1}$. 
We then prove that every $(d,d)$-sparse graph with minimum degree at most $d+1$ and maximum degree at most $d+2$ is independent in $\ell_p^d$. Finally, we prove that every triangulation of the projective plane is minimally rigid in $\ell_p^3$. A catalogue of rigidity preserving graph moves is also provided for the more general class of strictly convex and smooth normed spaces and we show that every triangulation of the sphere is independent for 3-dimensional spaces in this class.  
\end{abstract}

\maketitle
\tableofcontents


\section{Introduction}
Triangles, as everyone knows, are structurally rigid in the Euclidean plane, as are tetrahedral frames in Euclidean $3$-space, or the $1$-skeleton of any $d$-simplex in $d$-dimensional Euclidean space. In fact these are examples of minimally rigid structures since the removal of any edge will result in a flexible structure. More generally, one can consider the structural properties of \emph{bar-joint frameworks} obtained by embedding the vertices of a graph $G$ in $\mathbb{R}^d$. Such a framework is \emph{rigid} if the only edge-length-preserving continuous motions of the vertices arise from isometries of $\mathbb{R}^d$. There is a long and abiding theory of rigidity with its origins in both the work of Cauchy on Euclidean polyhedra \cite{Cau}  and the work of Maxwell on stresses and strains in structures \cite{Max}.

Much of the modern theory of rigidity considers a linearisation known as {\em infinitesimal rigidity}, which leads into matroid theory, and concentrates on the generic behaviour of the underlying graph. Standard graph operations such as Henneberg moves and vertex splitting moves \cite{NR} provide a means of constructing further rigid structures in a fixed dimension $d$, whereas the coning operation applied to a rigid $d$-dimensional structure produces a rigid structure one dimension higher \cite{whi83}. 

But what happens if the underlying Euclidean metric is changed? An illustrative example is the observation by Cook, Lovett and Morgan \cite{clm} that in any non-Euclidean normed plane a rhombus with generic diagonal lengths cannot be fully rotated whilst maintaining the distances between the corners.
The study of rigidity for graphs placed in non-Euclidean finite dimensional normed spaces was initiated by Kitson and Power \cite{kit-pow} (see also \cite{dew2,dew1,Ki} eg.). These works include the fundamental result, analogous to the Geiringer-Laman theorem for the Euclidean plane \cite{laman,pol}, that the minimally rigid graphs in dimension 2 are exactly those that decompose into the edge-disjoint union of two spanning trees.

Throughout this article we consider $d$-dimensional $\ell_p$-space (denoted $\ell_p^d$), where $p\in(1,\infty)$ and $p\not=2$, and occasionally the more general class of strictly convex and smooth normed spaces. 
In Section \ref{s:rigidity}, we provide some necessary background material and present the sparsity conjecture (Conjecture \ref{con:d-dim}) which is our main motivation for the sections that follow. Our first main result is in Section \ref{s:DH} where we provide a tight analogue of coning, which we term {\em bracing}, to transfer rigidity from $\ell_p^d$ to $\ell_p^{d+1}$ for certain complete graphs (Theorem \ref{t:bracing}). Using this, we show that for a non-Euclidean $\ell_p^d$-space with $p \in (1,\infty)$ and $p \neq 2$, the analogue of a $d$-simplex is the complete graph on $2d$ vertices, in the sense that it is minimally rigid for $\ell_p^d$ and there is no smaller graph with this property. In Section \ref{sec:ops}, we present several simple construction moves for generating new rigid structures from existing ones in a strictly convex and smooth space. 

Our second main result concerns independence which, as in the Euclidean case, is characterised by the rigidity matrix (defined below) having full rank. Analogous to a result for Euclidean frameworks due to Jackson and Jord\'{a}n \cite{JJbounded}, we obtain a result showing independence in  smooth non-Euclidean $\ell_p$-spaces for graphs of bounded degree (Theorem \ref{thm:main}). 

Our final main result concerns the rigidity of triangulated surfaces in dimension 3. It is well-known that the graph of a triangulated sphere is minimally rigid in the Euclidean space $\ell_2^3$ and that, in general, triangulations of closed surfaces are generically rigid in $\ell_2^3$ (see \cite{Fog,Glu} eg). It follows from Euler's formula that if $G=(V,E)$ is a triangulation of the sphere then $|E|=3|V|-6$ while if $G$ is a triangulation of any closed surface of orientable genus $>0$ then $|E|\geq 3|V|$. A graph which is minimally rigid for a non-Euclidean $\ell_p^3$-space must satisfy $|E|=3|V|-3$ and so such triangulations are clearly either underbraced or overbraced for $\ell_p^3$. Triangulations of the projective plane, on the other hand, do satisfy the necessary counting condition for minimal rigidity in non-Euclidean $\ell_p^3$-spaces and we prove that these triangulations are indeed minimally rigid (Theorem \ref{thm:projplane}).


\section{Rigidity in \texorpdfstring{$\ell_p^d$}{lp}}
\label{s:rigidity}
Let $X$ be a finite dimensional real normed linear space and let $X^*$ denote the dual space of $X$. Let $G=(V,E)$ be a finite simple graph with vertex set $V$, and consider a point $p=(p_v)_{v\in V}\in X^{V}$ such that the components $p_v$ and $p_w$ are distinct for each edge $vw\in E$. We refer to $p$ as a {\em placement} of the vertices of $G$ in $X$. The pair $(G,p)$ is referred to as a {\em bar-joint framework} in $X$. 

A linear functional $f: X \rightarrow \mathbb{R}$ \emph{supports} a non-zero point $x_0 \in X$ if $f(x_0) = \|x_0\|^2$ and $\sup_{\|x\|\leq 1}\,|f(x)|=\|x_0\|$; if exactly one linear functional supports a non-zero point $x_0$ then we say $x_0$ is \emph{smooth} and define $\varphi_{x_0}$ to be the unique support functional for $x_0$. A space $X$ is said to be \emph{smooth} if every non-zero point in $X$ is smooth. A space $X$ is said to be \emph{strictly convex} if $\|x+y\|<\|x\|+\|y\|$ whenever $x,y\in X$ are non-zero and $x$ is not a scalar multiple of $y$ (or equivalently, if the closed unit ball in $X$ is strictly convex). We will make use of the following elementary facts (see for example   \cite[Part III]{beauzamy} and \cite[Ch.~II]{cio} for a general treatment of these topics).

\begin{lemma}
\label{l:smu}
Let $X$ be a finite dimensional normed linear space and let $\mathcal{S}(X)$ denote the set of all smooth points in $X$ together with the point $0 \in X$.
Define $\Gamma : \mathcal{S}(X) \rightarrow X^*$ by setting $\Gamma(x) = \varphi_x$ and $\Gamma(0)=0$. Then,
\begin{enumerate}[(i)]
\item \label{l:smu1} $\Gamma$ is continuous,
\item \label{l:smu2} $X$ is strictly convex if and only if $\Gamma$ is injective,
\item \label{l:smu3} $X$ is smooth if and only if $\Gamma$ is surjective, and,
\item \label{l:smu4} $X$ is both strictly convex and smooth if and only if $\Gamma:X \rightarrow X^*$ is a homeomorphism.
\end{enumerate}
\end{lemma}

\subsection{Configuration spaces}
Two bar-joint frameworks $(G,p)$ and $(G,p')$ in $X$ are said to be {\em equivalent} if $\|p_v-p_w\|= \|p_v'-p_w'\|$ for each edge $vw\in E$.
The {\em configuration space} for $(G,p)$, denoted $\C(G,p)$, consists of all placements $p'\in X^{V}$ such that $(G,p')$ is equivalent to $(G,p)$.
The configuration space will always contain all translations of $p$,
however rotations and reflections of $p$ are not guaranteed to be contained in $\C(G,p)$,
as such operations do not always preserve distance in general normed spaces.
We can alternatively express the configuration space in terms of the {\em rigidity map}, 
\[f_G:X^{V}\to \bR^{E},\,\,\,\,\,\, (x_v)_{v\in V}\mapsto (\|x_v-x_w\|)_{vw\in E},\]
where we note that $\C(G,p) = f_G^{-1}(f_G(p))$.

Remembering that an isometry is a map from $X$ to itself that preserves the distance between points with respect to the norm of $X$,
a pair of frameworks $(G,p)$ and $(G,p')$ are said to be {\em isometric} if there exists an isometry $T:X\to X$ such that $p_v=T(p'_v)$ for all $v\in V$.
It is immediate that any two isometric frameworks will be equivalent,
but the converse is not true in general.
The set of placements $p'\in X^{V}$ such that $(G,p')$ is isometric to $(G,p)$ is denoted $\O_p$ (note this set depends only on $p$).
It can be shown (see \cite[Lemma 3.4]{dew1} for example) that $\O_p$ is a smooth submanifold of $X^{V}$.

\begin{remark}
If $X$ is the standard Euclidean $d$-space,
then a pair $(G,p)$ and $(G,p')$ are isometric if and only if the frameworks $(K,p)$ and $(K,p')$ are equivalent, where $K$ denotes the complete graph on the vertex set of $G$.
This is not however true in general for non-Euclidean normed spaces;
see \cite[Section 5]{dew1} for more discussion surrounding the topic.
\end{remark}

\subsection{The rigidity matrix}
Suppose $(G,p)$ is a bar-joint framework in a normed space $X$ with the property that $p_v-p_w$ is smooth in $X$ for each edge $vw\in E(G)$. Such placements $p$ are said to be {\em well-positioned} in $X$.
Given a basis $b_1,\ldots,b_d$ for $X$, the {\em rigidity matrix} for $(G,p)$ is a matrix $R(G,p) = (r_{e,(v,k)})$, with rows indexed by $E$ and columns indexed by $V\times\{1,\ldots,d\}$. 
The entries are defined as follows;
\[r_{e,(v,k)}=\left\{\begin{array}{cc}
\varphi_{p_v-p_w}(b_k) & \mbox{ if } e=vw,  \\
0 & \mbox{ otherwise}.
\end{array}
\right.\]

If the rank of $R(G,p)$ is maximal with respect to the set of all well-positioned placements of $G$ in $X$ then $(G,p)$ is said to be a \emph{regular} bar-joint framework.
If the rigidity matrix $R(G,p)$ has independent rows then $(G,p)$ is said to be {\em independent} in $X$.

\begin{remark}
Note that if the set $\S(X)$ of smooth points in a normed space $X$ is open then the set $\Reg(G;X)$ of regular placements of a graph $G=(V,E)$ in $X$ is an open subset of $X^V$. This follows immediately from Lemma \ref{l:smu}(i) and the fact that the rank function is lower semicontinuous.
\end{remark}

\subsection{Framework rigidity}
A regular bar-joint framework $(G,p)$ is  {\em rigid} in $X$ if the equivalent conditions of Proposition \ref{p:rigid} are satisfied.
\begin{proposition}{\cite[Theorem 1.1]{dew1}}
\label{p:rigid}
Let $(G,p)$ be a regular bar-joint framework in a finite dimensional real normed linear space $X$.
If $\S(X)$ is an open subset of $X$ then the following statements are equivalent.
\begin{enumerate}[(i)]
\item If $\gamma:[0,1]\to \C(G,p)$ is a continuous path with $\gamma(0)=p$ and $\gamma(1)=p'$ then $(G,p)$ and $(G,p')$ are isometric.
\item There exists an open neighbourhood $U$ of $p$ in $\C(G,p)$ such that if $p'\in U$ then $(G,p)$ and $(G,p')$ are isometric.
\item $\rank R(G,p) = d|V|-\dim \T(p)$, where $\T(p)$ denotes the tangent space of the smooth manifold $\O_p$ at $p$.  
\end{enumerate}
\end{proposition}
If a  bar-joint framework $(G,p)$ is both rigid and independent then it is said to be {\em minimally rigid} in $X$. A graph $G=(V,E)$ is said to be {\em independent} (respectively, {\em minimally rigid} or \emph{rigid}) in $X$  if there exists a placement $p \in X^{V}$ such that the pair $(G,p)$ is an independent (respectively, minimally rigid or rigid) bar-joint framework in $X$.

\subsection{Frameworks in $\ell^d_q$} 

\begin{remark}
	In the area of functional analysis,
	the variable $p$ is classically used to discuss $\ell_p$ spaces.
	This conflicts, however, with the standard notation for rigidity theory,
	where $p$ is usually used to denote a placement of a framework.
	To remove any ambiguity,
	we will from now on opt for $\ell_q^d$ spaces instead of $\ell_p^d$,
	and will retain $p$ for referring solely to placements.
\end{remark}

Let $\ell^d_q$ denote the $d$-dimensional vector space $\bR^d$ together with the norm 
$\|(x_1,\ldots,x_d)\|_q :=  (\sum_{k=1}^d |x_k|^q )^{\frac{1}{q}}$ 
where $d\geq1$ and $q\in (1,\infty)$. 
With respect to the usual basis on $\bR^d$, the rigidity matrix $R(G,p)$ for a bar-joint framework $(G,p)$ in $\ell^d_q$ has entries,
\[r_{e,(v,k)}=\left\{\begin{array}{cc}
\frac{\left[(p_v-p_w)^{(q-1)}\right]_k}{\|p_v-p_w\|_q^{q-2}} & \mbox{ if } e=vw, \\
0 & \mbox{ otherwise}.
\end{array}
\right.\]
Here, for convenience, we use the notation $x^{(q)} := (\sgn(x_1)|x_1|^q,\ldots,\sgn(x_d)|x_d|^q)$ and $[x]_k := x_k$ for each $x=(x_1,\ldots,x_d)\in \bR^d$.  
Note that by scaling each row of the rigidity matrix by the appropriate value $\|p_v-p_w\|_q^{q-2}$ we obtain an equivalent matrix $\tilde{R}(G,p)$ with entries,
\[r_{e,(v,k)}=\left\{\begin{array}{cc}
\left[(p_v-p_w)^{(q-1)}\right]_k & \mbox{ if } e=vw, \\
0 & \mbox{ otherwise}.
\end{array}
\right.\]
We refer to $\tilde{R}(G,p)$ as the {\em altered rigidity matrix} for $(G,p)$. 
It can be shown (see \cite[Lemma 2.3]{kit-pow}) that if $q\not=2$ then  $\dim \T(p) = d$. Thus, for $q\not=2$, a regular bar-joint framework $(G,p)$ in $\ell^d_q$ is rigid if and only if $\rank R(G,p) = d|V|-d$.

\begin{example}\label{ex:wheel}
Let $G$ be the wheel graph on vertices $V=\{v_0,v_1,v_2,v_3,v_4\}$ with center $v_0$ and let $q\in (1,\infty)$. Define $p$ to be the placement of $G$ in $\ell_q^2$ where,
\begin{align*}
p_{v_0}= (0,0), \quad p_{v_1}= (-1,0), \quad p_{v_2}= (0,1), \quad p_{v_3}= (1,0), \quad p_{v_4}= (1,-1).
\end{align*}
See left hand side of Figure \ref{fig:wheel}  for an illustration. The altered rigidity matrix $\tilde{R}(G,p)$ is as follows,
\[\kbordermatrix{
 &(v_0,1)&(v_0,2)&(v_1,1)&(v_1,2)&(v_2,1)&(v_2,2)&(v_3,1)&(v_3,2)&(v_4,1)&(v_4,2)\\
v_0v_1 &1&0&-1&0&0&0&0&0&0&0\\
v_0v_2 &0&-1&0&0&0&1&0&0&0&0\\
v_0v_3 &-1&0&0&0&0&0&1&0&0&0\\
v_0v_4 &-1&1&0&0&0&0&0&0&1&-1\\
v_1v_2 &0&0&-1&-1&1&1&0&0&0&0\\
v_2v_3 &0&0&0&0&-1&1&1&-1&0&0\\
v_3v_4 &0&0&0&0&0&0&0&1&0&-1\\
v_1v_4 &0&0&-2^{q-1}&1&0&0&0&0&2^{q-1}&-1\\
}.\]
Let $M$ be the $8 \times 8$ matrix formed by the first $8$ columns. We compute $\det M = 2^{q-1}-2$ and so, for $q\neq 2$, $\rank \tilde{R}(G,p) = 8 = 2|V|-2$.
Thus $(G,p)$ is regular and minimally rigid in $\ell_q^2$ for all $q \neq 2$.
Note that if we instead set $p_{v_4}=(0,-1)$ then the resulting bar-joint framework  is non-regular in $\ell_q^2$ for all $q \neq 2$ (see right hand side of Figure \ref{fig:wheel}).
\end{example}

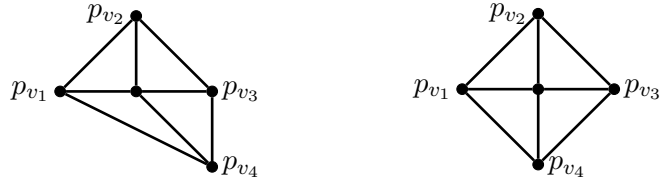
\begin{figure}[ht]
	\begin{tikzpicture}
		\node[vertex] (1) at (-1,0) {};
		\node[vertex] (2) at (0,1) {};
		\node[vertex] (3) at (1,0) {};
		\node[vertex] (4) at (1,-1) {};
		\node[vertex] (0) at (0,0) {};

		\node [left] at (-1,0) {$p_{v_1}$};
		\node [left] at (0,1.02) {$p_{v_2}$};
		\node [right] at (1,0) {$p_{v_3}$};
		\node [right] at (1,-1) {$p_{v_4}$};

		\draw[edge] (1)edge(2);
		\draw[edge] (2)edge(3);
		\draw[edge] (1)edge(4);
		\draw[edge] (3)edge(4);
		
		\draw[edge] (0)edge(1);
		\draw[edge] (0)edge(2);
		\draw[edge] (0)edge(3);
		\draw[edge] (0)edge(4);		
	\end{tikzpicture}
	\hspace{15mm}
		\begin{tikzpicture}
		\node[vertex] (1) at (-1,0) {};
		\node[vertex] (2) at (0,1) {};
		\node[vertex] (3) at (1,0) {};
		\node[vertex] (4) at (0,-1) {};
		\node[vertex] (0) at (0,0) {};
		
		\node [left] at (-1,0) {$p_{v_1}$};
		\node [left] at (0,1) {$p_{v_2}$};
		\node [right] at (1,0) {$p_{v_3}$};
		\node [right] at (0,-1.03) {$p_{v_4}$};
		
		\draw[edge] (1)edge(2);
		\draw[edge] (2)edge(3);
		\draw[edge] (1)edge(4);
		\draw[edge] (3)edge(4);
		
		\draw[edge] (0)edge(1);
		\draw[edge] (0)edge(2);
		\draw[edge] (0)edge(3);
		\draw[edge] (0)edge(4);		
	\end{tikzpicture}
	\caption{A bar-joint framework in $\ell_q^2$ which is regular and minimally rigid (left) and a bar-joint framework which is non-regular (right), for $q\in(1,\infty)$, $q\not=2$. }
	\label{fig:wheel}
\end{figure}

\subsection{The sparsity conjecture}
Given a graph $G=(V,E)$ and $d\geq 1$ we write $f_d(G) = d|V|-|E|$. 
We say $G$ is  \emph{$(d,d)$-sparse} if $f_d(H) \geq d$ for all subgraphs $H \subset G$.
If $G$ is $(d,d)$-sparse and $f_d(G) = d$ then $G$ is said to be \emph{$(d,d)$-tight}.

\begin{conjecture}\label{con:d-dim}
Let $q\in(1,\infty)$, $q\not=2$, and let $d\geq 1$. 
A graph $G$ is independent in $\ell_q^d$ if and only if $G$ is $(d,d)$-sparse.
\end{conjecture}

The conjecture above is a reformulation of a conjecture from \cite[Remark 3.16]{kit-pow}.
When $d=1$ the conjecture is true and the result is well-known. The case $d=2$ is proved in \cite{kit-pow} and is analogous to a landmark theorem proved independently by Pollaczek-Geiringer \cite{pol} and Laman \cite{laman} for graphs in the Euclidean plane. 
For $d\geq 3$, it is known that graphs which are independent in $\ell_q^d$
are necessarily $(d,d)$-sparse (see \cite{kit-pow}). Thus, it remains to prove the converse statement: every $(d,d)$-sparse graph is independent in $\ell_q^d$ for all $q\in (1,\infty)$, $q\not=2$, and for all $d\geq 3$. 

In this article, we prove this converse statement holds in three special cases: 1) when $|V|\leq 2d$, 2) when $G$ has minimum degree at most $d+1$ and maximum degree at most $d+2$, and 3) when $d=3$ and $G$ is a triangulation of the projective plane. We also provide a catalogue of independence preserving graph operations, including the well-known Henneberg moves, vertex splitting and rigid subgraph substitution.

\section{Dimension hopping}
\label{s:DH}
In this section we consider two graph operations called {\em coning} and {\em bracing}.
It is well-known that the coning operation preserves both independence and minimal rigidity when passing from $\ell_2^d$ to $\ell_2^{d+1}$ (see \cite{whi83}).
We will show that for $q\in(1,\infty)$, $q\not=2$, both the coning operation and the bracing operation preserve independence (but not minimal rigidity) when passing from $\ell_q^d$ to $\ell_q^{d+1}$. A simple application of the coning operation is that the complete graph $K_{d+1}$ is minimally rigid in $\ell_2^d$ for all $d\geq 2$. Indeed, $K_2$ is minimally rigid in 1-space, and for every $d\geq 2$, $K_{d+1}$ is obtained from $K_d$ by a coning operation. We will apply the bracing operation to prove the analogous result that $K_{2d}$ is minimally rigid in $\ell_q^d$, for all $d\geq 2$ and all $q\in(1,\infty)$, $q\not=2$. In particular, Conjecture \ref{con:d-dim} is true whenever $G$ is a subgraph of $K_{2d}$.

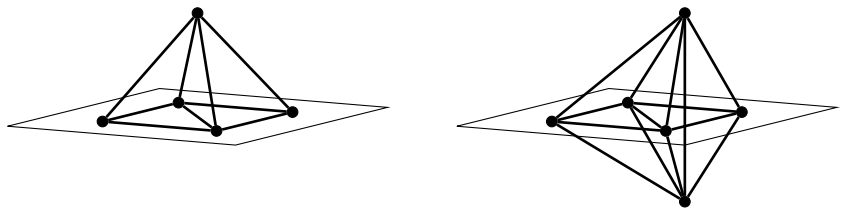
\begin{figure}[ht]
	\begin{tikzpicture}
		\node[vertex] (1) at (1.25,0.0625) {};
		\node[vertex] (2) at (2.25,0.3125) {};
		\node[vertex] (3) at (2.75,-0.0625) {};
		\node[vertex] (4) at (3.75,0.1875) {};
		\node[vertex] (0) at (2.5,1.5) {};
		\node[blankvertex] (0') at (2.5,-1) {};
		
		\draw (0,0) -- (2,0.5) -- (5,0.25) -- (3,-0.25) -- (0,0);
				
		\draw[edge] (1)edge(2);
		\draw[edge] (1)edge(3);
		\draw[edge] (2)edge(3);
		\draw[edge] (2)edge(4);
		\draw[edge] (3)edge(4);
		
		\draw[edge] (0)edge(1);
		\draw[edge] (0)edge(2);
		\draw[edge] (0)edge(3);
		\draw[edge] (0)edge(4);		
	\end{tikzpicture}
	\qquad 
	\begin{tikzpicture}
		\node[vertex] (1) at (1.25,0.0625) {};
		\node[vertex] (2) at (2.25,0.3125) {};
		\node[vertex] (3) at (2.75,-0.0625) {};
		\node[vertex] (4) at (3.75,0.1875) {};
		\node[vertex] (0) at (3,1.5) {};
		\node[vertex] (0') at (3,-1) {};
		
		\draw (0,0) -- (2,0.5) -- (5,0.25) -- (3,-0.25) -- (0,0);
				
		\draw[edge] (1)edge(2);
		\draw[edge] (1)edge(3);
		\draw[edge] (2)edge(3);
		\draw[edge] (2)edge(4);
		\draw[edge] (3)edge(4);
		
		\draw[edge] (0)edge(1);
		\draw[edge] (0)edge(2);
		\draw[edge] (0)edge(3);
		\draw[edge] (0)edge(4);		
		
		\draw[edge] (0')edge(1);
		\draw[edge] (0')edge(2);
		\draw[edge] (0')edge(3);
		\draw[edge] (0')edge(4);		
		
		\draw[edge] (0)edge(0');		
	\end{tikzpicture}
	\caption{Left: A coning operation applied to $K_4-e$. 
	Right: A bracing operation applied to $K_4-e$.}
	\label{fig:cone}
\end{figure}

\subsection{The coning operation}
\label{ss:coning}

Let $G=(V,E)$ and define $G'=(V',E')$ to be the graph with vertex set $V' = V \cup \{v_0\}$ and edge set $E'=E \cup \{v_0v : v \in V \}$.
Then $G'$ is said to be obtained from $G$ by a {\em coning operation}.
(See left hand side of Figure \ref{fig:cone} for an illustration).

\begin{theorem}\label{t:cone2}
Let $q\in (1,\infty)$ and let $d\geq 1$.
Suppose $G'=(V',E')$ is obtained from a graph $G=(V,E)$ by a coning operation.
If $G$ is independent in $\ell_q^d$ then $G'$ is independent in $\ell_q^{d+1}$.
\end{theorem}

\proof
Choose a placement $p$ such that $(G,p)$ is independent in $\ell_q^d$.
Let $\eta:\ell_q^d\to \ell_q^{d+1}$ be the natural embedding $(x_1,\ldots,x_d)\mapsto (x_1,\ldots,x_d,0)$.
Choose any $x\in \ell_q^{d+1}$ such that $\left[x \right]_{d+1}\not=0$.
Define $p'$ to be the placement of $G'$ in $\ell_q^{d+1}$ with $p'_v = \eta(p_v)$ 
for all $v \in V$ and $p'_{v_0} =x$.
Let $\omega= (\omega_{e})_{e \in E'}$ be a vector in the cokernel of  $\tilde{R}(G',p')$.
Then, for each $v\in V$ we have,
\[\omega_{v v_0} \left[(\eta(p_v)-x)^{(q-1)}\right]_{d+1}
=\sum_{w \in N_{G'}(v)} \omega_{vw} \left[(p'_v-p'_w)^{(q-1)}\right]_{d+1}
= 0.\]
Thus $\omega_{v v_0}=0$ for all $v \in V$ and so it follows that the vector $(\omega_{e})_{e \in E}$ lies in the cokernel of  $\tilde{R}(G,p)$. 
Since $\tilde{R}(G,p)$ is independent, we have $\omega_{e}=0$ for all $e\in E$. 
Hence, $\omega=0$ and so $(G',p')$ is independent in $\ell_q^{d+1}$. 
\endproof

\subsection{The bracing operation}
\label{s:bracing}
Let $d\geq 1$ and let $G=(V,E)$ be a finite simple graph with $|V|\geq 2d$.
Define $\tilde{G}$ to be the graph with vertex set $V(\tilde{G})=V\cup \{v_0,v_1\}$ and edge set,
\[E(\tilde{G})= E\cup\{v_0w:w\in S\}\cup \{v_1w:w\in S\}\cup\{v_0v_1\},\]
where $S\subseteq V$ and $|S|=2d$.
The graph $\tilde{G}$ is said to be obtained from $G$ by a {\em bracing operation} on $S$.
(See right hand side of Figure \ref{fig:cone} for an illustration).

\begin{lemma}
Let $G=(V,E)$ be a graph with $|V|\geq 2d$ and suppose $\tilde{G}$ is obtained from $G$ by a bracing operation on $S\subseteq V$, where $|S|=2d$.
\begin{enumerate}[(i)]
\item If $G$ is $(d,d)$-sparse  then $\tilde{G}$ is $(d+1,d+1)$-sparse.
\item If $G$ is $(d,d)$-tight then $\tilde{G}$ is $(d+1,d+1)$-tight if and only if $G=K_{2d}$.
\end{enumerate}
\end{lemma}

\proof
$(i)$ Let $\tilde{H}$ be a subgraph of $\tilde{G}$ and let $H=\tilde{H} \cap G$.  
Recall that $K_{2(d+1)}$ is $(d+1,d+1)$-sparse and so we may assume that $|V(\tilde{H})|>2d+2$.
If $\tilde{H}=H$ then,
\[|E(\tilde{H})| \leq d|V(\tilde{H})|-d
= (d+1)|V(\tilde{H})|-(d+1) - |V(\tilde{H})|+1.\]
If $|V(\tilde{H})| = |V(H)|+1$ then,
\begin{eqnarray*}
|E(\tilde{H})| &\leq& (d|V(H)|-d)+|S|
= (d+1)|V(\tilde{H})|-(d+1) - |V(\tilde{H})|+d+1. 
\end{eqnarray*}
Similarly, if $|V(\tilde{H})| = |V(H)|+2$ then,
\begin{eqnarray*}
|E(\tilde{H})| &\leq&  d|V(H)|-d+2|S|+1 
= (d+1)|V(\tilde{H})|-(d+1) - |V(\tilde{H})|+2d+2. 
\end{eqnarray*}
Thus $\tilde{G}$ is $(d+1,d+1)$-sparse.

$(ii)$ By a counting argument similar to $(i)$, $\tilde{G}$ is $(d+1,d+1)$-tight if and only if 
$|V(\tilde{G})|=2d+2$. In the latter case, $G$ is a $(d,d)$-tight graph with $|V|=2d$ and so $G=K_{2d}$. 

\endproof

\begin{theorem}
\label{t:bracing}
Let $G=(V,E)$ be a graph with $|V|\geq 2d$ and suppose $\tilde{G}$ is obtained from $G$ by a bracing operation on $S\subseteq V$, where $|S|=2d$.
Let $q\in (1,\infty)$, $q\not=2$, and let $d\geq 1$. 
If $G$ is independent in $\ell^d_q$ then $\tilde{G}$ is independent in $\ell^{d+1}_q$.
\end{theorem}

\proof
Let $p:V\to\bR^d$ be a  placement of $G$ in $\bR^d$ and write $p_w=(p_w^1,\ldots,p_w^d)$ for each $w\in V$.
Define $\tilde{p}:V(\tilde{G}) \to \bR^{d+1}$ by setting $\tilde{p}_w=(p_w^1,\ldots,p_w^d,0)$ for all $w\in V$, $\tilde{p}_{v_0}=(0,\ldots,0,-\lambda)$ and $\tilde{p}_{v_1}=(1,\ldots,1,\lambda)$ for some positive scalar $\lambda>0$.
Thus the vertices of $G$ are embedded in $\bR^d\times\{0\}$ and the two new vertices $v_0$ and $v_1$ are placed on the hyperplanes $x_{d+1}=-\lambda$ and $x_{d+1}=\lambda$ respectively. After a suitable permutation of rows and columns, the (altered) rigidity matrix for  $(\tilde{G},\tilde{p})$ takes the form,
\[\tilde{R}(\tilde{G},\tilde{p}) = \left[\begin{array}{cc}
\tilde{R}(G,p)&0 \\
* & D(p) 
\end{array}\right]\]
where $D(p)$ is a $(2|S|+1)\times (|V|+2(d+1))$-matrix. 
We will show that $D(p)$ is independent for some (and hence almost every) choice of $p$.

Suppose $|V|=2d$. Then $S = V$ and the rows of $D(p)$ are those indexed by the sets $E_0=\{v_0w:w\in V\}$ and $E_1=\{v_1w:w\in V\}$ together with the edge $v_0v_1$.
The columns of $D(p)$ are those indexed by $\{(w,d+1):w\in V\}$ together with the pairs $(v_0,1),\ldots,(v_0,d+1)$ and $(v_1,1),\ldots,(v_1,d+1)$. 
Thus, after a suitable permutation of rows and columns, $D(p)$ takes the form,
\[\kbordermatrix{
& (V;d+1) & & (v_0;1,\ldots,d) & & (v_1;1,\ldots,d) & & (v_0,d+1) & & (v_1,d+1)\\ 
E_0 & \begin{matrix}\lambda^{q-1}&&\\&\ddots&\\&&\lambda^{q-1}\end{matrix} &\vrule& D_{0}(p) &\vrule& 0 &\vrule& \begin{array}{c}-\lambda^{q-1}\\ \vdots \\-\lambda^{q-1}\end{array} &\vrule&\begin{array}{c}0\\ \vdots \\0\end{array}\\ \cline{2-10}
E_1 & \begin{matrix}-\lambda^{q-1}&&\\&\ddots&\\&&-\lambda^{q-1}\end{matrix} &\vrule& 0 &\vrule& D_{1}(p) &\vrule& \begin{array}{c}0\\ \vdots \\0\end{array} &\vrule& \begin{array}{c}\lambda^{q-1}\\ \vdots \\\lambda^{q-1}\end{array} \\ \cline{2-10}
v_0v_1 & 0 &\vrule& \begin{array}{ccc}-1 &\cdots&-1\end{array} &\vrule&\begin{array}{ccc}1 &\cdots&1\end{array} &\vrule& -(2\lambda)^{q-1} &\vrule& (2\lambda)^{q-1}}.\]
Note that  to show $D(p)$ is independent for some $p$, it is sufficient to show that the square submatrix of $D(p)$ formed by deleting the  $(v_1,d+1)$-column is independent. 
Adding each $E_0$ row, indexed by $v_0w$, of this square submatrix to the corresponding $E_1$ row, indexed by $v_1w$, we obtain, 
\[ \kbordermatrix{
& (V(G);d+1) & & (v_0;1,\ldots,d) & & (v_1;1,\ldots,d) & & (v_0,d+1) \\ 
E_0 & \begin{matrix}\lambda^{q-1}&&\\&\ddots&\\&&\lambda^{q-1}\end{matrix} &\vrule& D_{0}(p) &\vrule& 0 &\vrule& \begin{array}{c}-\lambda^{q-1}\\ \vdots \\-\lambda^{q-1}\end{array} 
\\ \cline{2-8}
E_1 & 0 &\vrule& D_0(p) &\vrule& D_{1}(p) &\vrule& \begin{array}{c}-\lambda^{q-1}\\ \vdots \\-\lambda^{q-1}\end{array}  \\ \cline{2-8}
v_0v_1 & 0 &\vrule& \begin{array}{ccc}-1 &\cdots&-1\end{array} &\vrule&\begin{array}{ccc}1 &\cdots&1\end{array} &\vrule& -(2\lambda)^{q-1}}.\]
It is clear that the first $|V|$ rows, indexed by $E_0$, are independent and that it is now sufficient to show there exists $p$ such that the $(2d+1)\times(2d+1)$-matrix, 
\[ A:=\kbordermatrix{
&  (v_0;1,\ldots,d) & & (v_1;1,\ldots,d) & & (v_0,d+1) \\ 
E_1 & D_0(p) &\vrule& D_{1}(p) &\vrule& \begin{array}{c}-\lambda^{q-1}\\ \vdots \\-\lambda^{q-1}\end{array} \\ \cline{2-6}
v_0v_1 & \begin{array}{ccc}-1 &\cdots&-1\end{array} &\vrule&\begin{array}{ccc}1 &\cdots&1\end{array} &\vrule& -(2\lambda)^{q-1}}\]
is independent. To this end, let $V=\{w_1,\ldots,w_{d},\tilde{w}_1,\ldots,\tilde{w}_d\}$. For each $i=1,\ldots,d$, choose $p_{w_i}\in \bR^d$ and   $p_{\tilde{w}_i}\in \bR^d$ such that, 
\[p_{w_i}^j 
= \left\{\begin{array}{cl} 0 & \mbox{ if }i=j,\\
\frac{1}{2} & \mbox{ otherwise},
\end{array}\right. \quad\quad \mbox{and } \quad\quad
p_{\tilde{w}_i}^j 
= \left\{\begin{array}{cl} 
1 & \mbox{ if }i=j,\\
\frac{1}{2} & \mbox{ otherwise}.
\end{array}\right.\]
Then, after a suitable permutation of rows and columns, 
the square submatrix $A$ takes the form,
\[ \kbordermatrix{
&  (v_0;1,\ldots,d) & & (v_1;1,\ldots,d) & & (v_0,d+1) \\ 
{\tiny \begin{array}{c}v_1w_1\\ \vdots \\ v_1w_d\end{array}} & 2I-2C &\vrule& 2C-I &\vrule& 
\begin{array}{c}-\lambda^{q-1}\\ \vdots \\-\lambda^{q-1}\end{array} \\ \cline{2-6}
{\tiny \begin{array}{c}v_1\tilde{w}_1\\ \vdots \\ v_1\tilde{w}_d\end{array}} & I-2C &\vrule& 2C-2I &\vrule& \begin{array}{c}-\lambda^{q-1}\\ \vdots \\-\lambda^{q-1}\end{array} \\ \cline{2-6}
v_0v_1 & \begin{array}{ccc}-1 &\cdots&-1\end{array} &\vrule&\begin{array}{ccc}1 &\cdots&1\end{array} &\vrule& -(2\lambda)^{q-1}}\]
where $C$ is the $d\times d$-matrix,
\[C=\left[\begin{array}{cccc}
 1 &\frac{1}{2^{q-2}} & \cdots &\frac{1}{2^{q-2}}  \\  
 \frac{1}{2^{q-2}}& 1 & \cdots &\frac{1}{2^{q-2}} \\
 \vdots & \vdots&\ddots & \vdots  \\
 \frac{1}{2^{q-2}}&\frac{1}{2^{q-2}} & \cdots &1
\end{array}\right].\]
Subtracting each $v_1\tilde{w}_i$ row from the corresponding $v_1w_i$ row and applying further row reductions, 
this matrix reduces to,
\[ B:=\kbordermatrix{
&  (v_0;1,\ldots,d) & & (v_1;1,\ldots,d) & & (v_0,d+1) \\ 
{\tiny \begin{array}{c}v_1w_1\\ \vdots \\ v_1w_d\end{array}} & I &\vrule& I &\vrule& 
\begin{array}{c}0\\ \vdots \\0\end{array} \\ \cline{2-6}
{\tiny \begin{array}{c}v_1\tilde{w}_1\\ \vdots \\ v_1\tilde{w}_d\end{array}} & 0 &\vrule& C &\vrule& \begin{array}{c}-\lambda^{q-1}\\ \vdots \\-\lambda^{q-1}\end{array} \\ \cline{2-6}
v_0v_1 & \begin{array}{ccc}0 &\cdots&0\end{array} &\vrule&\begin{array}{ccc}2 &\cdots&2\end{array} &\vrule& -(2\lambda)^{q-1}}.\]
For $i=1,\ldots,d$, let $r_{i}$ denote the row of $B$ which is indexed by $v_1\tilde{w}_i$ and let $r_e$ denote the row indexed by $v_0v_1$.
Note that $C$ is a circulant matrix with determinant,
\[\det(C) =\left(1+\frac{d-1}{2^{q-2}}\right)\left(1-\frac{1}{2^{q-2}}\right)^{d-1}.\]
Thus, since $q\not=2$, $C$ is invertible and so the rows $r_1,\ldots,r_d$ are independent.

Suppose $r_{e} = \sum_{i=1}^d \mu_{i} r_{i}$ for some scalars $\mu_{1},\ldots, \mu_d\in \bR$. 
On considering the $(v_0,d+1)$ column it is clear that $\sum_{i=1}^d\mu_i =2^{q-1}$. 
Moreover, considering the $(v_1,1)$ column,
\[
\left(1-\frac{1}{2^{q-2}}\right)\mu_{1}
= \mu_1+\frac{1}{2^{q-2}}(\mu_2+\cdots+\mu_d)-2\\
=0.\]
Thus, since $q\not=2$, we have $\mu_1=0$. By similar arguments, $\mu_2=\cdots=\mu_d=0$.
Thus the matrix $B$, and hence also the matrices $A$ and $D(p)$, are independent.

Note that the set of points $p$ for which $D(p)$ is independent is open and dense in $\bR^{d|V(G)|}$. Thus we may choose $p\in \bR^{d|V(G)|}$ such that both $\tilde{R}(G,p)$ and $D(p)$ are independent. In particular, $\tilde{R}(\tilde{G},\tilde{p})$ is independent, as required. 

Finally, if $|V|>2d$ then note that $\tilde{R}(\tilde{G},\tilde{p})$ will have additional columns, indexed by $\{(w,d+1):w\in V\backslash S\}$, with zero entries. These columns do not alter the dependencies between the rows and so the result follows as above. 
\endproof

As a corollary we show that $K_{2d}$ is minimally rigid in $\ell_q^d$ for $q\in (1,\infty)$, $q\not=2$.

\begin{corollary}\label{cor:K2drigid}
\label{c:K2d}
Let $q\in (1,\infty)$, $q\not=2$. 
\begin{enumerate}[(i)]
\item If $|V|\leq 2d$ then $G=(V,E)$ is independent in $\ell^d_q$.
\item $K_{2d}$ is minimally rigid in $\ell^d_q$ for all $d\geq 1$.
\end{enumerate}
\end{corollary}

\proof
It is clear that $K_2$ is independent in $\bR$.
Note that, for all $d\geq 2$, $K_{2d}$ is obtained from $K_{2(d-1)}$ by a bracing operation on the vertex set of $K_{2(d-1)}$.
Thus, by Theorem \ref{t:bracing}, $K_{2d}$ is independent in $\ell^d_q$ for all $d\geq2$.
If $|V|\leq 2d$ then $G$ is a subgraph of $K_{2d}$ and hence is independent in $\ell^d_q$. 
Finally, since $K_{2d}$ is independent and $|E|=d|V|-d$, it is minimally rigid in $\ell^d_q$.
\endproof

\begin{remark}
We conjecture that $K_{2d}$ admits a rigid  (but not necessarily minimally rigid) placement in every $d$-dimensional normed space. This conjecture clearly holds for the Euclidean norm and the above corollary confirms the conjecture for all non-Euclidean smooth $\ell_q$ norms. The conjecture is also known to hold for all non-Euclidean normed planes (see \cite{dew2}) and for the cylinder and hypercylinder norms on $\bR^3$ and $\bR^4$ respectively (see \cite{kitlev17}).
\end{remark}

\section{Graph operations}
\label{sec:ops}

In this section we provide a catalogue of graph operations which preserve independence in smooth and strictly convex normed spaces. These include the well known Henneberg moves (0 and 1-extensions), vertex splitting moves and rigid subgraph substitutions. By applying any sequence of these graph operations to $K_{2d}$ we may obtain a large class of minimally rigid graphs  for $\ell_q^d$ when $q\in(1,\infty)$ and $q\not=2$.

\subsection{0-extensions}

\begin{definition}
Let $G=(V,E)$ be a graph and define $G'$ by setting $V(G')=V\cup \{v\}$ and $E(G')= E\cup\{vw:w\in S\}$,
where $S\subseteq V$ and $|S|=d$.
The graph $G'$ is said to be obtained from $G$ by a {\em $d$-dimensional 0-extension} on $S$; see Figure \ref{fig:0}.
\end{definition}

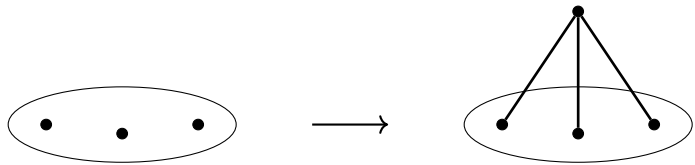
\begin{figure}[ht]
	\begin{tikzpicture}
		\node[vertex] (1) at (-5,0) {};
		\node[vertex] (2) at (-6,-0.12) {};
		\node[vertex] (3) at (-7,0) {};
		
		\draw (-6,0) ellipse (1.5cm and 0.5cm);
		
		\draw[thick,->] (-3.5,0) -- (-2.5,0);
		
		\node[vertex] (1') at (1,0) {};
		\node[vertex] (2') at (0,-0.12) {};
		\node[vertex] (3') at (-1,0) {};
		\node[vertex] (0') at (0,1.5) {};
		
		\draw (0,0) ellipse (1.5cm and 0.5cm);
				
		\draw[edge] (0')edge(1');
		\draw[edge] (0')edge(2');
		\draw[edge] (0')edge(3');
	\end{tikzpicture}
	\caption{A $3$-dimensional $0$-extension.}
	\label{fig:0}
\end{figure}

To prove that 0-extensions preserve rigidity in the generality of strictly convex and smooth normed spaces we will need the following lemma.

\begin{lemma}
\label{l:qhom2}
Let $X$ be a finite dimensional real normed linear space which is smooth and strictly convex  and let $d=\dim X$. Let $y_1, \ldots, y_n \in X$ where $n \leq d$.
Then, for all $\epsilon > 0$, there exists $y'_1, \ldots, y'_n \in X$ such that $\|y_i - y'_i \| < \epsilon$ for each $1 \leq i \leq n$ and $\varphi_{y'_1}, \ldots, \varphi_{y'_n}$ are linearly independent in $X^*$.
\end{lemma}

\proof
Let $\epsilon>0$ and let $b_1,\ldots,b_d$ be a basis for $X$.
Define,
\[\theta:X^{n}\to M_{n \times d}(\bR), \,\,\,\,
(x_1,\ldots,x_n)\mapsto \left[\begin{array}{ccc}\varphi_{x_1}(b_1) & \cdots & \varphi_{x_1}(b_d) \\
\vdots & & \vdots\\
\varphi_{x_n}(b_1) & \cdots & \varphi_{x_n}(b_d)\end{array}\right].\]
Note that since $X$ is smooth and strictly convex then by Lemma \ref{l:smu}(\ref{l:smu4}), the duality map $\Gamma:X\to X^*$, $x\mapsto \varphi_x$, is a homeomorphism. It follows that $\theta$ is also a homeomorphism. 
Recall that the set $\I_{n\times d}(\bR)$ of independent $n\times d$ real matrices is open and dense in $M_{n\times d}(\bR)$. 
Thus $\theta^{-1}(\I_{n\times d}(\bR))$ is dense in $X^{n}$
and so there exists $y'=(y'_1, \ldots, y'_n) \in X^{n}$ such that $\|y_i - y'_i \| < \epsilon$ for each $1 \leq i \leq n$ and $\theta(y')$ is independent.
In particular, the linear functionals $\varphi_{y'_1}, \ldots, \varphi_{y'_n}$ are linearly independent, as required.
\endproof

In the following proposition, the set of regular placements of $G$ in $X$ is denoted $\text{Reg}(G;X)$.

\begin{proposition}
\label{p:0ext}
Let $X$ be a finite dimensional real normed linear space which is smooth and strictly convex and let $d= \dim X$.
Let $G=(V,E)$ be a graph and suppose $G'$ is obtained from $G$ by a $d$-dimensional $0$-extension on $S\subseteq V$, where $|S|= d$. Then $G$ is independent (resp. minimally rigid) in $X$ if and only if $G'$ is independent (resp. minimally rigid) in $X$.
\end{proposition}

\proof
Let $\theta$ be the homeomorphism described in Lemma \ref{l:qhom2} for $n=d$.
Then $\theta^{-1}(\GL_d(\bR))$ is dense in $X^{d}$ where $\GL_d(\bR)$ denotes the general linear group of degree $d$ over $\bR$.
Note that the map, 
\[ \eta: X^{|V|}\to M_{|E|\times|V|}(\bR), \quad x \mapsto R(G,x),\]
is continuous. Since the rank function is lower semicontinuous, it follows that $\Reg(G;X)$ is open in $X^{|V|}$.
Thus the intersection 
\[ \Reg(G;X)\cap (X^{|V\backslash S|}\times \theta^{-1}(\GL_d(\bR))) \]
is  non-empty  in $X^{|V|}$.
Let $p=(p_1,p_2)$ be a point in this intersection and set 
\[ p'=(p_1,p_2,0)\in X^{|V\backslash S|}\times X^{d}\times X. \]
Here $p'$ describes a placement of $G'$ in $X$ in which $p_1$ is a placement of the vertices in $V\backslash S$, $p_2$ is a placement of the vertices in $S$, and the new vertex $v$ is placed at the origin.
After a suitable permutation of rows and columns, the rigidity matrix for $(G',p')$ takes the form,
\[R(G',p') = \begin{bmatrix}
R(G,p)&0 \\
C(p_2)&\theta(p_2)
\end{bmatrix}.\]
As $p_1\in  \Reg(G;X)$ and $\theta(p_2)$ is invertible it follows that $p'\in \Reg(G';X)$. 
Note that $R(G',p')$ is independent if and only if $R(G,p)$ is independent, and that $f_d(G')= f_d(G)$ so the result follows.
\endproof

\subsection{1-extensions}

\begin{definition}
Let $G=(V,E)$ be a graph containing vertices $v_1,\ldots,v_{d+1}$ and the edge $v_dv_{d+1}\in E$. 
Define $G'$ by setting, 
\begin{align*}
V(G')=V\cup \{v_0\}, \qquad E(G')= (E\backslash \{v_dv_{d+1}\})\cup\{v_0v_1,\ldots,v_0v_{d+1}\}.
\end{align*}
The graph $G'$ is said to be obtained from $G$ by a {\em $d$-dimensional 1-extension} on the vertices $v_1,\ldots,v_{d+1}\in V$ and the edge $v_dv_{d+1}\in E$; see Figure \ref{fig:1}.  
\end{definition}

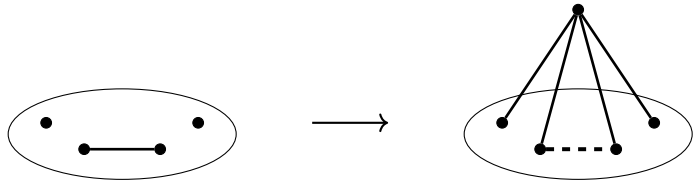
\begin{figure}[ht]
	\begin{tikzpicture}
		\node[vertex] (1) at (-5,0) {};
		\node[vertex] (3) at (-7,0) {};
		
		\node[vertex] (4) at (-5.5,-0.35) {};
		\node[vertex] (5) at (-6.5,-0.35) {};
		
		\draw[edge] (4)edge(5);
		
		\draw (-6,-0.15) ellipse (1.5cm and 0.6cm);
		
		\draw[thick,->] (-3.5,0) -- (-2.5,0);
		
		\node[vertex] (1') at (1,0) {};
		\node[vertex] (3') at (-1,0) {};
		
		\node[vertex] (0') at (0,1.5) {};
		
		\node[vertex] (4') at (-0.5,-0.35) {};
		\node[vertex] (5') at (0.5,-0.35) {};
		
		\draw[dashedge] (4')edge(5');
		
		\draw (0,-0.15) ellipse (1.5cm and 0.6cm);
				
		\draw[edge] (0')edge(1');
		\draw[edge] (0')edge(3');
		\draw[edge] (0')edge(4');
		\draw[edge] (0')edge(5');
		
	\end{tikzpicture}
	\caption{An example of a $3$-dimensional $1$-extension.}
	\label{fig:1}
\end{figure}

\begin{proposition}
\label{p:1ext}
Let $X$ be a finite dimensional real normed linear space which is smooth and strictly convex and let $d= \dim X$.
Suppose $G'$ is obtained from $G=(V,E)$ by a $d$-dimensional $1$-extension.
If $G$ is independent in $X$
then $G'$ is independent in $X$. 
Further; if both $G$ and $G'$ are independent then $G$ is rigid in $X$ if and only if $G'$ is rigid in $X$.
\end{proposition}

\begin{proof}
Let $v_0$ be the unique vertex in $V(G') \setminus V$, let $v_0 v_1, \ldots, v_0 v_{d+1} \in E(G')$ be the added edges for distinct $v_1, \ldots ,v_{d+1} \in V$, and let $v_d v_{d+1}$ be the deleted edge. 
If $G$ is independent in $X$ then there exists a placement $p$ of $G$ in $X$ for which $(G,p)$ is independent.
By translating the framework $(G,p)$ we may assume without loss of generality that 
$p_v\not=0$ for all $v\in V$ and $p_{v_{d+1}} = - p_{v_d}$.
By Lemma \ref{l:qhom2}, and since the set of independent placements of $G$ is open in $X^{V}$, we may also assume that the linear functionals $\varphi_{p_{v_1}}, \ldots, \varphi_{p_{v_d}}$ are linearly independent.
Define a placement $p'$ of $G'$ in $X$ by setting $p'_v=p_v$ for all $v\in V$ and
$p'_{v_0} =0$. We claim that $(G',p')$ is independent in $X$.

Suppose $a=(a_{e})_{e\in E(G')} \in \mathbb{R}^{E(G')}$ is a linear dependence on the rows of $R(G',p')$. From the entries of the $v_0$-column of $R(G',p')$ we obtain,
\[ \sum_{i=1}^{d-1} a_{v_0 v_i}  \varphi_{p_{v_i}} +(a_{v_0 v_{d}} - a_{v_0 v_{d+1}}) \varphi_{p_{v_d}}  
= -\sum_{i=1}^{d+1} a_{v_0 v_i} \varphi_{p'_{v_0} - p'_{v_i}}  
= 0.\]
Thus, since $\varphi_{p_{v_1}}, \ldots, \varphi_{p_{v_d}}$ are linearly independent,  
we have $a_{v_0 v_1} = \ldots = a_{v_0 v_{d-1}} = 0$ and $a_{v_0 v_{d}} =  a_{v_0 v_{d+1}}$.  
Define $b=(b_{e})_{e\in E} \in \mathbb{R}^{E}$ with $b_{e} = a_{e}$ for $e \neq v_d v_{d+1}$ and $b_{v_d v_{d+1}} = \frac{1}{2}a_{v_0 v_{d}}$. Then $b$ is a linear dependence on the rows of $R(G, p)$.
Thus $b=0$ as $(G, p)$ is independent. 
It now follows that $a=0$ and so $(G',p')$ is independent, as required.

The final statement of the proposition follows since $f_d(G') = f_d(G)$. 
\end{proof}

\subsection{Vertex splitting}

\begin{definition}\label{vertexsplitdef}
Let $G=(V,E)$ be a graph containing a vertex $v_0 \in V$ and edges $v_0 v_i \in E$ for $i=1, \ldots, d-1$. 
Let $G'$ be a graph obtained from $G$ by the following process:
\begin{enumerate}[(i)]
\item adjoin a new vertex $w_0$ to $G$ together with the edges $w_0 v_0, w_0 v_1, \ldots, w_0 v_{d-1}$, 
\item for every edge of the form $v_0 w$ in $E$, where $w \notin\{v_1,\ldots, v_{d-1}\}$, either leave the edge as it is or replace it with the edge $w_0 w$. 
\end{enumerate}
The graph $G'$ is said to be obtained from $G$ by a {\em d-dimensional vertex split} at the vertex $v_0\in V$ and edges $v_0 v_1, \ldots, v_0 v_{d-1}\in E$; see Figure \ref{fig:vsplit}.  
\end{definition}

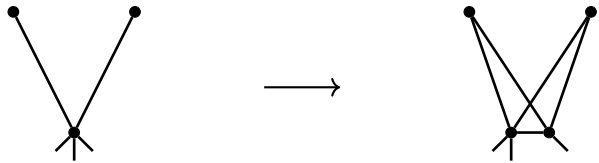
\begin{figure}[ht]
	\begin{tikzpicture}
		\node[vertex] (0) at (-6,0) {};
		\node[vertex] (1) at (-6.8,1.6) {};
		\node[vertex] (2) at (-5.2,1.6) {};
		
		\node[blankvertex] (n1) at (-6.3,-0.3) {};
		\node[blankvertex] (n2) at (-6,-0.45) {};
		\node[blankvertex] (n3) at (-5.7,-0.3) {};
		
		\draw[edge] (0)edge(1);
		\draw[edge] (0)edge(2);	

		\draw[edge] (0)edge(n1);
		\draw[edge] (0)edge(n2);	
		\draw[edge] (0)edge(n3);	
			
		\draw[thick,->] (-3.5,0.6) -- (-2.5,0.6);
		
		\node[blankvertex] (m1) at (-0.55,-0.3) {};
		\node[blankvertex] (m2) at (-0.25,-0.45) {};
		\node[blankvertex] (m3) at (0.55,-0.3) {};

		\node[vertex] (0') at (-0.25,0) {};
		\node[vertex] (0'') at (0.25,0) {};
		\node[vertex] (1') at (-0.8,1.6) {};
		\node[vertex] (2') at (0.8,1.6) {};
		
		\draw[edge] (0')edge(1');
		\draw[edge] (0')edge(2');
		\draw[edge] (0'')edge(1');
		\draw[edge] (0'')edge(2');
		\draw[edge] (0')edge(0'');
		
		\draw[edge] (0')edge(m1);
		\draw[edge] (0')edge(m2);	
		\draw[edge] (0'')edge(m3);	
	\end{tikzpicture}
	\caption{A $3$-dimensional vertex split.}
	\label{fig:vsplit}
\end{figure}

For a graph $G=(V,E)$ and a vertex $v\in V$, we will use $N_G(v)$, or $N(v)$ when the context is clear, to denote the set of neighbours of $v$ in $G$.

\begin{proposition}
\label{p:split}
Let $X$ be a smooth and strictly convex normed space with dimension $d$.
Suppose $G'$ is a $d$-dimensional vertex split of $G$. If $G$ is independent in $X$ then $G'$ is independent in $X$. Further; if both $G$ and $G'$ are independent then $G'$ is rigid in $X$ if and only if $G$ is rigid in $X$.
\end{proposition}

\begin{proof}
Let $v_0,w_0,v_1, \ldots, v_{d-1}$ be as described in Definition \ref{vertexsplitdef}. 
Since $G$ is independent in $X$ there exists a placement $p\in X^V$ of $G$ in $X$ such that $R(G,p)$  is independent. 
Choose $y \in X\backslash\{0\}$. 
By Lemma \ref{l:qhom2}, and since the set of independent placements of $G$ is open in $X^V$, we may assume that the linear functionals $\varphi_y, \varphi_{p_{v_0} - p_{v_1}}, \ldots, \varphi_{p_{v_0} - p_{v_{d-1}}}$ are linearly independent. 
Write $E(G')=E_1\cup E_2 \cup \{v_0w_0\}$ where $E_1$ consists of all edges in $G'$ which are not incident with $w_0$ and $E_2$ consists of all edges in $G'$ of the form $vw_0$ with $v \neq v_0$.
Fix a basis $b_1, \ldots,b_d$ for $X$ and define $R$ to be the $|E(G')| \times d|V(G')|$ matrix with non-zero row entries as described below and zero entries everywhere else,
\begin{align*}
\bbordermatrix{ &  & \scriptstyle{(v,i)} &   & \scriptstyle{(w,i)} &  & \scriptstyle{(v_0,i)} &  & \scriptstyle{(w_0,i)} &   \cr	
\scriptscriptstyle{vw \in E_1} & \ldots & \varphi_{p_v-p_w}(b_i) & \ldots & -\varphi_{p_v-p_w}(b_i) & \ldots & \ldots & \ldots & \ldots & \ldots \cr
\scriptscriptstyle{v w_{0} \in E_2} & \ldots & \varphi_{p_v-p_{v_0}}(b_i) & \ldots & \ldots & \ldots & \ldots & \ldots & -\varphi_{p_v-p_{v_0}}(b_i) & \ldots \cr
\scriptscriptstyle{v_0 w_{0}} & \ldots & \ldots & \ldots & \ldots & \ldots & \varphi_y(b_i) & \ldots & -\varphi_y(b_i) & \ldots \cr}
\end{align*}
Suppose $a \in \mathbb{R}^{E(G')}$ is a linear dependence on the rows of $R$.
Define $b \in \mathbb{R}^{E}$, where 
\begin{align*}
b_{vw} := 
\begin{cases}
a_{v_0 v_i} + a_{w_0 v_i} & \text{if } v w = v_0 v_i \text{ for any } i =1, \ldots, d-1, \\
a_{v w_0} & \text{if } v w = v v_0 \text{ but } v v_0 \notin E(G'),\\
a_{vw} & \text{otherwise.}
\end{cases}
\end{align*}
If $v\not= v_0$ then note that 
$\sum_{w \in N_G(v)} b_{v w} \varphi_{p_v-p_w} = \sum_{w \in N_{G'}(v)} a_{v w} \varphi_{p_v-p_w}=0.$
Also note that $\sum_{w \in N_G(v_0)} b_{v w} \varphi_{p_{v_0}-p_w}=A+B$ where,
\[A=a_{v_0 w_0} \varphi_y + \sum_{w \in N_{G'}(v_0) \setminus \{w_0\}} a_{v_0 w} \varphi_{p_{v_0}-p_w}  =0,\]
\[B=- a_{v_0 w_0} \varphi_y + \sum_{w \in N_{G'}(w_0) \setminus \{v_0\}} a_{w_0 w} \varphi_{p_{v_0}-p_w}=0.\]
Thus if $b\not=0$ then $b$ is a linear dependence on the rows of $R(G,p)$, a contradiction. We conclude that $b=0$.
In particular,  we have $a_{v_0 v_i}= -a_{w_0 v_i}$ for all $i=1,\ldots,d-1$ and  
$a_{vw}=0$ for all edges $vw$ in $E(G')\backslash \{v_0w_0,v_0v_i,w_0v_i:i=1,\ldots,d-1\}$.
As $\varphi_y,\varphi_{p_{v_0} - p_{v_1}}, \ldots, \varphi_{p_{v_0} - p_{v_{d-1}}}$ are linearly independent then by observing how the linear dependence acts on the $v_0$ columns of $R$ we obtain $a_{v_0w_0}=0$ and $a_{v_0v_i}=0$ for all $i=1,\ldots, d-1$. Thus $a=0$ and so $R$ is independent.

Let $\epsilon>0$ and let $R_\epsilon$ denote the independent matrix obtained by multiplying the entries of the $v_0w_0$ row of $R$ by $\epsilon$. Define a placement $p'$ of $G'$ in $X$ by setting $p'_v=p_v$ for all $v\in V$ and $p'_{w_0} =p_{v_0} +\epsilon y$. 
Note that for each edge $v_0w\in E(G)$, $p_{v_0}-p_w$ is a smooth point of $X$. Thus, using Lemma \ref{l:smu}, it follows that for $\epsilon$ sufficiently small, the rigidity matrix $R(G',p')$ will lie in an open neighbourhood of $R_\epsilon$ consisting of independent matrices. We conclude that $G'$ is independent in $X$.

The final statement of the proposition follows since $f_d(G') = f_d(G)$. 
\end{proof}

\begin{remark}
There is a natural variant of vertex splitting known as \emph{spider splitting}. In this version, $d$ vertices adjacent to $v_0$ become adjacent to both $v_0$ and $w_0$ but there is no edge between $v_0$ and $w_0$, see Figure \ref{fig:spsplit}.  With a simplified version of the proof of Proposition \ref{p:split} above we obtain the analogous result. The 2-dimensional spider split has been considered in Euclidean contexts under other names such as the vertex-to-4-cycle move  \cite{NR}.
\end{remark}

\begin{figure}[ht]
	\begin{tikzpicture}
		\node[vertex] (0) at (-6,0) {};
		\node[vertex] (1) at (-7,1.6) {};
		\node[vertex] (2) at (-5,1.6) {};
		\node[vertex] (3) at (-6,1.8) {};
		
		\node[blankvertex] (n1) at (-6.3,-0.3) {};
		\node[blankvertex] (n2) at (-6,-0.45) {};
		\node[blankvertex] (n3) at (-5.7,-0.3) {};
		
		\draw[edge] (0)edge(1);
		\draw[edge] (0)edge(2);
		\draw[edge] (0)edge(3);	

		\draw[edge] (0)edge(n1);
		\draw[edge] (0)edge(n2);	
		\draw[edge] (0)edge(n3);	
			
		\draw[thick,->] (-3.5,0.6) -- (-2.5,0.6);
		
		\node[blankvertex] (m1) at (-0.55,-0.3) {};
		\node[blankvertex] (m2) at (-0.25,-0.45) {};
		\node[blankvertex] (m3) at (0.55,-0.3) {};

		\node[vertex] (0') at (-0.25,0) {};
		\node[vertex] (0'') at (0.25,0) {};
		\node[vertex] (1') at (-1,1.6) {};
		\node[vertex] (2') at (1,1.6) {};
		\node[vertex] (3') at (0,1.8) {};
		
		\draw[edge] (0')edge(1');
		\draw[edge] (0')edge(2');
		\draw[edge] (0')edge(3');
		\draw[edge] (0'')edge(1');
		\draw[edge] (0'')edge(2');
		\draw[edge] (0'')edge(3');
		
		\draw[edge] (0')edge(m1);
		\draw[edge] (0')edge(m2);	
		\draw[edge] (0'')edge(m3);	
	\end{tikzpicture}
	\caption{A $3$-dimensional spider split.}
	\label{fig:spsplit}
\end{figure}
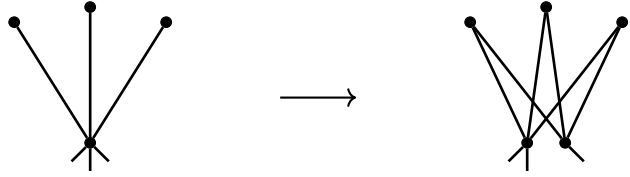

\subsection{Graph substitution} 

\begin{definition}
Let $G$ and $H$ be graphs and choose $v_0\in V(G)$. 
A graph $G'$ is obtained from $G$ by a {\em vertex-to-$H$ substitution} at $v_0$ if it is formed by replacing the vertex $v_0 \in V(G)$ with $V(H)$, adding the edges $E(H)$ and changing each edge $v_0w \in E(G)$ to $v w$ for some $v \in V(H)$. See Figure \ref{fig:sub} for an example of a vertex-to-$K_4$ substitution applied to a wheel graph.
\end{definition}

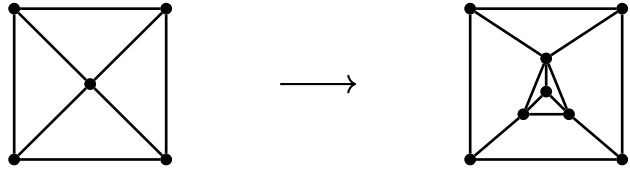
\begin{figure}[ht]
	\begin{tikzpicture}
		\node[vertex] (0) at (-6,0) {};
		\node[vertex] (1) at (-5,1) {};
		\node[vertex] (2) at (-5,-1) {};
		\node[vertex] (3) at (-7,-1) {};
		\node[vertex] (4) at (-7,1) {};
		
		\draw[edge] (0)edge(1);
		\draw[edge] (0)edge(2);
		\draw[edge] (0)edge(3);	
		\draw[edge] (0)edge(4);	
		
		\draw[edge] (1)edge(2);
		\draw[edge] (2)edge(3);
		\draw[edge] (3)edge(4);
		\draw[edge] (4)edge(1);
		
		\draw[thick,->] (-3.5,0) -- (-2.5,0);

		\node[vertex] (1') at (1,1) {};
		\node[vertex] (2') at (1,-1) {};
		\node[vertex] (3') at (-1,-1) {};
		\node[vertex] (4') at (-1,1) {};
		
		\node[vertex] (00) at (0,-0.1) {};
		\node[vertex] (01) at (0,0.34) {};
		\node[vertex] (02) at (0.3,-0.4) {};
		\node[vertex] (03) at (-0.3,-0.4) {};
		
		\draw[edge] (00)edge(01);
		\draw[edge] (00)edge(02);
		\draw[edge] (00)edge(03);
		\draw[edge] (01)edge(02);
		\draw[edge] (01)edge(03);
		\draw[edge] (02)edge(03);
				
		\draw[edge] (01)edge(1');
		\draw[edge] (02)edge(2');
		\draw[edge] (03)edge(3');	
		\draw[edge] (01)edge(4');	
		
		\draw[edge] (1')edge(2');
		\draw[edge] (2')edge(3');
		\draw[edge] (3')edge(4');
		\draw[edge] (4')edge(1');
	\end{tikzpicture}
	\caption{A vertex-to-$K_4$ substitution at the center vertex of the wheel graph on $5$ vertices. This graph operation will preserve rigidity in any non-Euclidean $2$-dimensional normed space \cite[Lemma 5.5]{dew2}.}
	\label{fig:sub}
\end{figure}

Recall that $\T(p)$ denotes the tangent space at $p$ of $\O_p$; the smooth manifold of placements isometric to $p$. Our next result shows that the vertex-to-$H$ substitution move preserves independence for a normed space $X$ whenever $H$ is independent. 

\begin{proposition}
\label{p:sub}
Let $X$ be a normed space with dimension $d$ and suppose that the set of smooth points of $X$ form an open subset.
Suppose $G'$ is obtained from $G$ by a vertex-to-$H$ substitution at $v_0$. 
If $G$ and $H$ are independent in $X$ then $G'$ is independent in $X$. 
Further; if $\dim \T(r) = d$ for any placement $r$ of $H$ and $H$ is rigid in $X$ then $G'$ is rigid in $X$ if and only if $G$ is rigid in $X$.
\end{proposition}

\begin{proof}
Let $(G,p)$ and $(H,r)$ be independent in $X$. Denote by $\partial V(H)$ all the edges in $G'$ with exactly one vertex in $V(H)$. Let $b_1, \ldots, b_d$ be a basis for $X$. Consider the $|E(G')| \times d|V(G')|$ matrix $R$ with non-zero row entries as described below,
\begin{align*}
\bbordermatrix{  &   &   &   & \scriptstyle{(v,i)} &   &   &   & \scriptstyle{(w,i)} &   &   &   \cr	
\scriptstyle{vw \in E(G) \cap E(G')}           & & \ldots &  & \varphi_{p_v-p_w}(b_i)        &  & \ldots  &  & -\varphi_{p_v-p_w}(b_i)         &  & \ldots &  \cr
\scriptstyle{vw \in E(H)}                      &  & \ldots &  & \varphi_{r_v-r_w}(b_i)      &  & \ldots  &  & -\varphi_{r_v-r_w}(b_i)         &  & \ldots &  \cr
\scriptstyle{vw \in \partial V(H), ~ v \in V(H)} &  & \ldots &  & \varphi_{p_{v_0}-p_w}(b_i)  &  & \ldots  &  & -\varphi_{p_{v_0}-p_w}(b_i)  &  & \ldots &  \cr}
\end{align*}
Suppose $a \in \mathbb{R}^{E(G')}$ is a linear dependence on the rows of $R$. Define $b \in \mathbb{R}^{E(G)}$ by setting $b_{vw} = a_{vw}$ if $vw \in E(G) \cap E(G')$ and $b_{v_0 w} = a_{vw}$ if the edge $v_0 w \in E(G)$ is replaced by $vw \in \partial V(H)$. 
If $v\notin N_G(v_0)\cup\{v_0\}$ then note that,
\[\sum_{w \in N_G(v)} b_{vw} \varphi_{p_{v}-p_w} = 
\sum_{w \in N_{G'}(v)} a_{vw} \varphi_{p_{v}-p_w} =0.\]
If $v\in N_G(v_0)$ and $vv_0\in E(G)$ is replaced by $vz \in \partial V(H)$ then note that,
\[\sum_{w \in N_G(v)} b_{vw} \varphi_{p_{v}-p_w} = 
a_{vz} \varphi_{p_{v}-p_{v_0}}+\sum_{w \in N_{G'}(v)\backslash \{z\}} a_{vw} \varphi_{p_{v}-p_w} =0.\]
Since,
\[\sum_{v \in V(H)}\sum_{w \in N_H(v)} a_{vw} \varphi_{r_v-r_w}
= \sum_{vw \in E(H)} a_{vw} (\varphi_{r_v-r_w} + \varphi_{r_w-r_v}) 
= 0,\]
we have, 
\begin{eqnarray*}
\sum_{w \in N_G(v_0)} b_{v_0 w} \varphi_{p_{v_0}-p_w} 
&=& \sum_{v \in V(H)}  \sum_{w \in N_{G'}(v) \setminus V(H)} a_{vw} \varphi_{p_{v_0}-p_w}\\
&=& \sum_{v \in V(H)} \left(\sum_{w \in N_H(v)} a_{vw} \varphi_{r_v-r_w}  + \sum_{w \in N_{G'}(v) \setminus V(H)} a_{vw} \varphi_{p_{v_0}-p_w}\right)\\
&=& 0.
\end{eqnarray*}
Thus, if $b\not=0$ then $b$ is a linear dependence on the rows of $R(G,p)$. Since $R(G,p)$ is independent, it follows that $b=0$. In particular, $a_{vw}=0$ for all $vw\in E(G')\backslash E(H)$.
Note that if $a_H=(a_{vw})_{vw \in E(H)}$ is non-zero then $a_H$ is a linear dependence on the rows of $R(H,r)$. Since $R(H,r)$ is independent, we conclude that $a_H=0$ and so $a=0$. Thus we have shown that $R$ is independent.

Let $\epsilon>0$ and let $R_\epsilon$ denote the independent matrix obtained by multiplying the entries of the $E(H)$ rows of $R$ by $\epsilon$. 
Define a placement $p'$ of $G'$ in $X$ by setting $p'_v=p_v$ for all $v\in V(G')\backslash V(H)$ and $p'_{v} =p_{v_0}+ \epsilon r_v$ for all $v\in V(H)$. 
Note that for each edge $v_0w\in E(G)$, $p_{v_0}-p_w$ is a smooth point of $X$. Thus, using Lemma \ref{l:smu}, it follows that for $\epsilon$ sufficiently small, the rigidity matrix $R(G',p')$ will lie in an open neighbourhood of $R_\epsilon$ consisting of independent matrices. We conclude that $G'$ is independent in $X$.

If $H$ is also rigid and $\dim \mathcal{T}(r)=d$ for any choice of placement $r$ of $H$, then we note that $f_d(G') = f_d(G)$, thus $G'$ is rigid if and only if $G$ is rigid. 
\end{proof}

\begin{remark}
It can be shown that Proposition \ref{p:sub} holds for any normed space. Since the proof is significantly more technical we refer the reader to \cite{dewarthesis} for details.
\end{remark}

\section{Degree-bounded graphs}

Recall that Conjecture \ref{con:d-dim} proposed a characterisation of independence in $\ell_q^d$. We will prove the conjecture for a certain family of degree bounded graphs. This is analogous to a theorem of Jackson and Jord\'{a}n \cite{JJbounded} who worked in the Euclidean space $\ell_2^d$. 

Let $G=(V,E)$. For $U\subset V$, let $G[U]$ denote the subgraph of $G$ induced by $U$ and let $i_G(U)$, or simply $i(U)$ when the context is clear, denote the number of edges in $G[U]$. We also use $d(U,W)$ to denote the number of edges of the form $xy$ with $x\in U\setminus W$ and $y\in W\setminus U$, where $U,W\subset V$.
Let $\delta(G)$ denote the minimum degree in the graph $G$ and $\Delta(G)$ denote the maximum degree in $G$. Let $d_G(v)$, or simply $d(v)$, denote the degree of a vertex $v$ in $G$.

\begin{theorem}\label{thm:main}
Let $q\in (1,\infty)$, $q\not=2$ and let $d\geq 3$.
Suppose $G$ is a connected graph with $\delta(G)\leq d+1$ and $\Delta(G)\leq d+2$ for any $d\geq 3$.
Then $G$ is independent in $\ell_q^d$ if and only if $G$ is $(d,d)$-sparse.
\end{theorem}

To prove the theorem we will need several additional lemmas.
The first of these is easily proved by counting the contribution to both sides.

\begin{lemma}\label{lem:4}
Let $G=(V,E)$. For any $U,W\subset V$ we have $i(U)+i(W)+d(U,W)=i(U\cup W)+i(U\cap W)$.
\end{lemma}

We will say that $U\subset V$ is \emph{critical} if $|U|>1$ and $i(U)=d|U|-d$. 

\begin{lemma}\label{lem:5}
Let $G=(V,E)$ be $(d,d)$-sparse and suppose $U\subset V$ is critical. Then $d_{G[U]}(v)\geq d$ for all $v\in U$.
\end{lemma}

\begin{proof}
Suppose $U$ is critical and there exists $x\in U$ with $d_{G[U]}(x)< d$. Then 
$$i(U-\{x\})=i(U)- d_{G[U]}(x)=d|U|-d- d_{G[U]}(x)=d|U-\{x\}|- d_{G[U]}(x)>d|U-\{x\}|-d,$$ contradicting the $(d,d)$-sparsity of $G$.
\end{proof}

Let $G=(V,E)$. A graph $G'$ is said to be obtained from $G$ by a \emph{($d$-dimensional) 1-reduction} at $v$ adding $x_1x_2$ if $V(G')=V-\{v\}$, for some vertex $v$ with $N_G(v)=\{x_1,x_2,\dots, x_{d+1}\}$, and $E(G')=E\setminus \{vx_1,vx_2,\dots,vx_{d+1} \}\cup \{x_1x_2\}$.

\begin{lemma}\label{lem:6}
Let $G=(V,E)$ be $(d,d)$-sparse, suppose $v\in V$ has $d(v)=d+1$ and $x,y\in N(v)$. Then the graph resulting from a 1-reduction at $v$ adding $xy$ is not $(d,d)$-sparse if and only if either $xy\in E$ or there exists a critical set $U$ with $x,y\in U\subset V-\{v\}$.
\end{lemma}

\begin{proof}
If $xy\in E$ or there exists a critical set $U$ with $x,y\in U\subset V-\{v\}$ then it is obvious that the 1-reduction at $v$ adding $xy$ does not result in a $(d,d)$-sparse graph. Conversely if a 1-reduction at $v$ adding $xy$ does not result in a $(d,d)$-sparse graph then either there is a pair of parallel edges between $x$ and $y$ in the resulting graph giving $xy\in E$ or there is a violation of $(d,d)$-sparsity. In the latter case let $G'$ be the graph resulting from the specified 1-reduction. Then there is a subgraph of $H_1=(V_1,E_1)$ of $G'$ with $i(V_1)=d|V_1|-(d-1)$. Clearly $x,y\in V_1$, otherwise $H_1$ is a subgraph of $G$ contradicting the $(d,d)$-sparsity of $G$. Hence $V_1$, as a subset of $V$, is the required critical set in $G$.
\end{proof}

The key technical lemma we will need is the following.

\begin{lemma}\label{lem:7}
Let $d\geq 3$ and suppose $G=(V,E)$ is $(d,d)$-sparse. Suppose $v\in V$ has $d(v)=d+1$ and $d(x)\leq d+2$ for all $x\in N(v)$. Then there is a 1-reduction at $v$ which results in a $(d,d)$-sparse graph unless $G[\{v\}\cup N(v)]=K_{d+2}$.
\end{lemma}

\begin{proof}
Suppose $G[\{v\}\cup N(v)]\neq K_{d+2}$. Then without loss of generality we may suppose that $xy\notin E$ for some $x,y\in N(v)$. Hence Lemma \ref{lem:6} implies there is a critical set $U\subset V-v$ with $x,y\in U$. Choose $U$ to be the maximal critical set containing $x,y$ but not $v$. If $N(v)\subset U$ then $i(U\cup \{v\})> d|U\cup \{v\}|-d$, contradicting $(d,d)$-sparsity. So without loss of generality we may suppose $w\notin U$ for some $w\in N(v)\setminus \{x,y\}$.

Suppose there is a critical set $W$ with $y,w\in W\subset V-\{v\}$. Then, by the maximality of $U$, $U\cup W$ is not critical, so $i(U\cup W)\leq d|U\cup W|-(d+1)$. Since $G$ is $(d,d)$-sparse we also have $i(U\cap W)\leq d|U\cap W|-d$.
Now using Lemma \ref{lem:4} we get $d|U|+d|W|-2d+d(U,W)\leq d|U\cup W|+d|U\cap W|-2d-1$, a contradiction.

Hence Lemma \ref{lem:6} implies that $yw\in E$. The same argument applied to the pair $x,w$ implies that $xw\in E$. Since $d\geq 3$, there exists $z\in N(v)\setminus \{x,y,w\}$. If $z\notin U$ then we can repeat the same argument to the pair $y,z$ to find that $yz\in E$. However this would imply that $d+2\geq d(y)\geq d_{G[U]}(y)+3$, which is a contradiction by Lemma \ref{lem:5}.
Hence for all $z\in N(v)\setminus \{x,y,w\}$ we have that $z\in U$. We may now apply the previous argument to each pair $z,w$ to see that each $zw\in E$. Hence $w$ has $d$ neighbours in $U$ so $U'=U\cup \{w\}$ is critical, contradicting the maximality of $U$.
\end{proof}

We can now prove the theorem.

\begin{proof}[Proof of Theorem \ref{thm:main}]
Necessity is easy. For the sufficiency we use induction on $|V|$. The base cases are $K_1$ and $K_{d+2}$. The latter of which is independent in $\ell_q^d$ by Corollary \ref{c:K2d}(i).

Suppose $G=(V,E)$ is $(d,d)$-sparse, $|V|\geq 2$, $G\neq K_{d+2}$ and $v\in V$ has minimum degree. Suppose first that $G-v$ is disconnected. Then each component $H_i=(V_i,E_i)$ of $G-v$ is connected with $\delta(H_i)\leq d+1$ and $\Delta(H_i)\leq d+2$. Hence $H_i$ is independent in $\ell_q^d$ by induction. Since $d_{H_i+v}(v)\leq d$, Proposition \ref{p:0ext} implies that $G[V_i+v]$ is independent in $\ell_q^d$. Hence $G$ is independent in $\ell_q^d$ by Proposition \ref{p:sub}.
Thus we may suppose that $G-v$ is connected.

Suppose $d(v)\leq d$. Then $G-v$ is connected with $\delta(G-\{v\})\leq d+1$ and $\Delta(G-\{v\})\leq d+2$. Hence $G-\{v\}$ is independent in $\ell_q^d$ by induction and $G$ is independent in $\ell_q^d$ by Proposition \ref{p:0ext}. Thus we may suppose that $d(v)=d+1$. Suppose $G[\{v\}\cup N(v)]\neq K_{d+2}$. Then Lemma \ref{lem:7} implies there is a 1-reduction at $v$ which results in a $(d,d)$-sparse graph $G'$. Since $G-\{v\}$ is connected, $G'$ is connected. Since $\delta(G)\leq d+1$ and $\Delta(G)\leq d+2$ we also have $\delta(G')\leq d+1$ and $\Delta(G')\leq d+2$. By induction $G'$ is independent in $\ell_q^d$ and hence $G$ is independent in $\ell_q^d$ by Proposition \ref{p:1ext}.

Hence $G[\{v\}\cup N(v)]= K_{d+2}$. Since $G\neq K_{d+2}$, there exists $u\in V\setminus V(K_{d+2})$. Consider $H=G-K_{d+2}$. Each component $H_i$ of $H$ is connected with $\delta(H_i)\leq d+1$ and $\Delta(H_i)\leq d+2$. Hence $H_i$ is independent in $\ell_q^d$ by induction, and trivially $H$ is independent in $\ell_q^d$. Note that for each vertex $r\in K_{d+2}$, there is at most one edge of the form $rs$ where $s\in H$. Thus $G$ is a subgraph of the graph formed from $K_{d+3}$ by a vertex-to-$H$ move on $t$ where $t$ is the vertex of $K_{d+3}$ not in the $K_{d+2}$. 
Also, since $d\geq 3$, $K_{d+3}$ is independent in $\ell_q^d$ by Corollary \ref{c:K2d}(i).
That $G$ is independent in $\ell_q^d$ now follows from Proposition \ref{p:sub}.
\end{proof}

We close this section by noting another independence result for normed spaces which we adapt from \cite{JJbounded}. This time we may use the combinatorics of \cite{JJbounded} directly.

\begin{theorem}\label{thm:ind2}
Let $X$ be a smooth and strictly convex normed space of dimension $3$ and let $G=(V,E)$ be a graph such that $i(U)\leq \frac{1}{2}(5|U|-7)$ for all $U\subset V$ with $|U|\geq 2$. Then $G$ is independent in $X$.
\end{theorem}

\begin{proof}
We use induction on $|V|$. If $|V|=2$ then trivially $K_2$ is independent in $X$. If $|V|\geq 3$ then, in the proof of \cite[Theorem 5.1]{JJbounded}, it was shown that there must exist a 0-reduction or a 1-reduction on $G$ to a smaller graph satisfying the hypotheses of the theorem. Since this smaller graph is independent in $X$ by induction the proof is completed by application of Propositions \ref{p:0ext} and  \ref{p:1ext}.
\end{proof}

Note that neither Theorem \ref{thm:main} nor \ref{thm:ind2} are best possible. Indeed if Conjecture \ref{con:d-dim} is true then one can remove the degree hypotheses in Theorem  \ref{thm:main} and replace the sparsity assumption in Theorem  \ref{thm:ind2} by $(3,k)$-sparsity, where $k$ is the dimension of the isometry group of the normed space. On the other hand it seems to be a difficult problem to work with vertices of degree 5 so even extending Theorem \ref{thm:ind2} to include the case when $i(U)=\frac{1}{2}5|U|$ may be challenging.

\section{Surface graphs}
In this final section we consider the graphs of triangulated surfaces.
We will use our results to deduce first that every triangulation of the sphere is independent in $\ell_q^3$ and then that every triangulation of the projective plane is minimally rigid in $\ell_q^3$ for $1<q\neq 2<\infty$.
To this end we will use the following topological results providing recursive constructions of triangulations of the sphere and of the projective plane by vertex splitting due to Steinitz \cite{Ste} and Barnette \cite{bar}. In the statements we use \emph{topological vertex splitting} to mean a vertex splitting operation that preserves the surface, and we use $K_7-K_3$ to denote the unique graph obtained from $K_7$ by deleting the edges of a triangle.

\begin{proposition}[\cite{Ste}]\label{prop:sphere}
Every triangulation of the sphere can be obtained from $K_4$ by topological vertex splitting operations.
\end{proposition}

\begin{proposition}[\cite{bar}]\label{prop:projplane}
Every triangulation of the projective plane can be obtained from $K_6$ or $K_7-K_3$ by topological vertex splitting operations.
\end{proposition}

\begin{theorem}
Let $X$ be a smooth and strictly convex normed space of dimension $3$, and let $G$ be a triangulation of the sphere.
Then $G$ is independent  in $X$.
\end{theorem}

\begin{proof}
Let $G$ be a triangulation of the sphere. Proposition \ref{prop:sphere} shows that $G$ can be generated from $K_4$ by vertex splitting operations.  
We may use Proposition \ref{p:0ext} to deduce that $K_4$ is indepedendent in $X$ and Proposition \ref{p:split} shows that vertex splitting preserves minimal rigidity in $X$. The theorem follows from these results by an elementary induction argument.
\end{proof}

To give an analogous result for the projective plane we will need to restrict to $\ell_q^3$ and make use of the following lemmas.

\begin{lemma}\label{l:powerinequality}
Let $x > y > 0$.
If $k > 1$ then $x^k - y^k > (x -y )^k$
and if $k<1$ then $x^k - y^k < (x -y )^k$.
\end{lemma}

\begin{proof}
Fix $y \in (0, \infty)$ and define the smooth function
$f : (y , \infty) \rightarrow \mathbb{R}, ~ t \mapsto t^k - y^k - (t - y)^k$.
We note that 
$f'(t)= k t^{k-1} - k (t- y)^{k-1}.$
If $k > 1$ then $f'(t) >0$ and $f$ is strictly increasing, while if $k < 1$ then $f'(t) <0$ and $f$ is strictly decreasing.
As $\lim_{t \rightarrow y} f(t) =0$,
it follows that if $k > 1$ then $f(t) >0$, while if $k < 1$ then $f(t) <0$.
The result now follows by choosing $ x >y $ and rearranging $f(x)$.
\end{proof}

\begin{lemma}\label{l:specialk4ind}
Let $q \in (1,2) \cup (2,\infty)$, let $\gamma \in (0,1)$ and
let $p^\gamma$ be the placement of the complete graph $K_4$ on the vertex set $\{v_0 , v_1 , v_2, v_3 \}$ with,
\[p^\gamma_{v_0}=(0,0),\quad p^\gamma_{v_1} = (0,1),\quad p^\gamma_{v_2} = (-1,0),
\quad p^\gamma_{v_3} = (\gamma, \gamma).\]
Then $(K_4,p^\gamma)$ is independent in $\ell_q^2$.
\end{lemma}

\begin{proof}
Consider the $6 \times 6$-matrix
\begin{align*}
M_\gamma : = \left[\begin{array}{cccccc}
0 & 1 & 0 & 0 & 0 & 0 \\
0 & 0 & -1 & 0 & 0 & 0 \\
0 & 0 & 0 & 0 & \gamma^{q-1} & \gamma^{q-1} \\
1 & 1 & -1 & -1 & 0 & 0 \\
-\gamma^{q-1} & (1-\gamma)^{q-1} & 0 & 0 & \gamma^{q-1} & -(1-\gamma)^{q-1} \\
 0 & 0 & -(1+\gamma)^{q-1} & -\gamma^{q-1} & (1+\gamma)^{q-1} & \gamma^{q-1} \\
\end{array}\right]
\end{align*}
Note that $M_\gamma$ is the submatrix of the altered $\tilde{R}(K_4,p^\gamma)$ formed by removing the columns corresponding to $v_0$.
Thus, if $M_\gamma$ is invertible then $(K_4,p^\gamma)$ is independent.
We have,
\begin{align*}
\det M_\gamma = (\gamma^{q-1})^2 \left(2 \gamma^{q-1} - (1+\gamma)^{q-1} + (1- \gamma)^{q-1} \right).
\end{align*}
By Lemma \ref{l:powerinequality}, if $q-1 >1$ then,
\begin{align*}
(1+\gamma)^{q-1} - (1-\gamma)^{q-1} > 2^{q-1} \gamma^{q-1} > 2\gamma^{q-1},
\end{align*}
while if $q-1 <1$ then,
\begin{align*}
(1+\gamma)^{q-1} - (1-\gamma)^{q-1} < 2^{q-1} \gamma^{q-1} < 2\gamma^{q-1}.
\end{align*}
Thus $\det M_\gamma\not=0$ and so $M_\gamma$ is invertible, as required.
\end{proof}

\begin{lemma}\label{lem:k7-k3}
The graph $K_7 - K_3$ is minimally rigid in $\ell_q^3$ for any $q \in (1,\infty), q\neq 2$.
\end{lemma}

\begin{proof}
Let $G:=K_7 - K_3$ be the graph with vertex set $V := \{ v_0 , v_1 , v_2, v_3 , a , b , c \}$ and edge set $E := K(V) \setminus \{ ab, ac , bc \}$.
Choose $\gamma \in (0,1)$.
We now define a placement $p$ of $G$ in $\ell_q^3$ by putting
\[p_{v_0}=(0,0,0),\quad p_{v_1} = (0,1,0),\quad p_{v_2} = (-1,0,0),\quad p_{v_3} = (\gamma, \gamma,0),\]
\[p_a = (0,0,-1),\quad  p_b = (1,1,1),\quad p_c = (1,0,1).\]
Let $(K_4,r)$ be the bar-joint framework in $\ell_q^2$ with, 
\[r_{v_0}=(0,0),\quad r_{v_1} = (0,1), \quad r_{v_2} = (-1,0), \quad r_{v_3} = (\gamma, \gamma).\]
Then, by Lemma \ref{l:specialk4ind}, the altered rigidity matrix $\tilde{R}(K_4,r)$ is independent.
By shifting all $(v_i;1)$ and $(v_i;2)$ columns of $\tilde{R}(G,p)$ to the left, we obtain the matrix
\begin{align*}
\left[\begin{array}{cc}
\tilde{R}(K_4,r)&0_{6 \times 13} \\
* & M 
\end{array}\right],
\end{align*}
where for any $a \in \mathbb{R}$ we define $a_{n \times m}$ to be the $n \times m$ matrix with $a$ for each entry,
and $M$ is a $12 \times 13$ matrix.
To show $\tilde{R}(G,p)$ is independent it suffices to show $M$ has row independence.

By reordering rows and columns if needed,
we have that
\begin{align*}
M = \kbordermatrix{
 & (v_0,z) \ldots (v_3,z) & (a,x) & (a,y) & (b,x) & (b,y) & (c,x) & (c,y) & (a,z) & (b, z) & (c, z) \\
v_i a, ~  0\leq i \leq 3 & I_4  & A_x & A_y& 0_{4 \times 1}& 0_{4 \times 1}& 0_{4 \times 1}& 0_{4 \times 1}& -1_{4 \times 1}& 0_{4 \times 1}& 0_{4 \times 1} \\
v_i b, ~ 0\leq i \leq 3 & -I_4 & 0_{4 \times 1} & 0_{4 \times 1} & B_x & B_y & 0_{4 \times 1}& 0_{4 \times 1}& 0_{4 \times 1}& 1_{4 \times 1}& 0_{4 \times 1} \\
v_i c,  ~0\leq i \leq 3 & -I_4 & 0_{4 \times 1} & 0_{4 \times 1} & 0_{4 \times 1}& 0_{4 \times 1}& C_x & C_y & 0_{4 \times 1}& 0_{4 \times 1}& 1_{4 \times 1} \\
}.
\end{align*}
(we order the rows $(v_0,a), \ldots, (v_3,a)$, etc.)
where $I_4$ is the $4 \times 4$ identity matrix and
\begin{align*}
A_x := \left[\begin{array}{c}
0 \\
0 \\
1 \\
-\gamma^{q-1}
\end{array}\right] ,
\qquad
A_y := \left[\begin{array}{c}
0 \\
-1 \\
0 \\
-\gamma^{q-1}
\end{array}\right]= C_y, \\
B_x := \left[\begin{array}{c}
1 \\
1 \\
2^{q-1} \\
(1-\gamma)^{q-1}
\end{array}\right]= C_x,
\qquad
B_y := \left[\begin{array}{c}
1 \\
0 \\
1 \\
(1-\gamma)^{q-1}
\end{array}\right] \, .
\end{align*}
By applying row operations to $M$ we obtain a $12 \times 13$ matrix of the form
\begin{align*}
\left[\begin{array}{cc}
I_4&* \\
0_{8 \times 4} & N
\end{array}\right],
\end{align*}
where $N$ is the $8 \times 9$ matrix
\begin{align*}
N: = \left[\begin{array}{ccccccccc}
A_x & A_y & B_x & B_y & 0_{4 \times 1}& 0_{4 \times 1}& -1_{4 \times 1}& 1_{4 \times 1}& 0_{4 \times 1} \\
A_x & A_y & 0_{4 \times 1}& 0_{4 \times 1}& C_x & C_y & -1_{4 \times 1}& 0_{4 \times 1}& 1_{4 \times 1} \\
\end{array}\right],
\end{align*}
and we note that the rows of $N$ are linearly independent if and only if the rows of $M$ are linearly  independent.
By adding the seventh and ninth columns to the eighth column
followed by subtracting the first four rows of $N$ from the last four rows of $N$ 
(i.e. subtract the first from the fifth, the second from the sixth, etc.)
we obtain
\begin{align*}
\left[\begin{array}{ccccccccc}
A_x & A_y & B_x & B_y & 0_{4 \times 1}& 0_{4 \times 1}& -1_{4 \times 1}& 0_{4 \times 1}& 0_{4 \times 1} \\
0_{4 \times 1} & 0_{4 \times 1} & -B_x & -B_y & C_x & C_y & 0_{4 \times 1}& 0_{4 \times 1}& 1_{4 \times 1} \\
\end{array}\right] \, .
\end{align*}
We may remove the eighth column to obtain the $8 \times 8$ matrix
\begin{align*}
O: = \left[\begin{array}{cccccccc}
0 & 0 & 1 & 1 & 0 & 0 & -1 & 0 \\
0 & -1 & 1 & 0 & 0 & 0 & -1 & 0 \\
1 & 0 & 2^{q-1} & 1 & 0 & 0 & -1 & 0 \\
-\gamma^{q-1} & -\gamma^{q-1} & (1-\gamma)^{q-1} & (1-\gamma)^{q-1} & 0 & 0 & -1 & 0 \\
0 & 0 & -1 & -1 & 1 & 0 & 0 & 1 \\
0 & 0 & -1 & 0 & 1 & -1 & 0 & 1 \\
0 & 0 & -2^{q-1} & -1 & 2^{q-1} & 0 & 0 & 1 \\
0 & 0 & -(1-\gamma)^{q-1} & -(1-\gamma)^{q-1} & (1-\gamma)^{q-1} & -\gamma^{q-1} & 0 & 1
\end{array}\right]
\end{align*}
and note that the rows of $N$ are linearly independent if and only if the rows of $O$ are linearly  independent.
By subtracting the first row from the second, third and fourth rows, and by subtracting the fifth row from the sixth, seventh and eighth rows,
followed by deleting the first and fifth rows and the last two columns,
we obtain the $6 \times 6$ matrix
\begin{align*}
P: = \left[\begin{array}{cccccc}
0 & -1 & 0 & -1 & 0 & 0 \\
1 & 0 & 2^{q-1}-1 & 0 & 0 & 0 \\
-\gamma^{q-1} & -\gamma^{q-1} & (1-\gamma)^{q-1} -1 & (1-\gamma)^{q-1} -1& 0 & 0 \\
0 & 0 & 0 & 1 & 0 & -1  \\
0 & 0 & -2^{q-1} +1 & 0 & 2^{q-1}-1 & 0 \\
0 & 0 & -(1-\gamma)^{q-1} +1 & -(1-\gamma)^{q-1} +1 & (1-\gamma)^{q-1} -1 & -\gamma^{q-1} 
\end{array}\right]
\end{align*}
and, as $\det P=\det O$, $O$ is invertible if and only if $P$ is invertible. 
By subtracting the second column of $P$ from the fourth, adding the sixth column of $P$ to the fourth,
and then deleting the second and sixth columns and the first and fourth rows,
we obtain the $4 \times 4$ matrix
\begin{align*}
Q: = \left[\begin{array}{cccc}
1 &  2^{q-1}-1 & 0 & 0  \\
-\gamma^{q-1}  & (1-\gamma)^{q-1} -1 & (1-\gamma)^{q-1} -1+\gamma^{q-1}& 0  \\
0  & -2^{q-1} +1 & 0 & 2^{q-1}-1  \\
0 & -(1-\gamma)^{q-1} +1 & -(1-\gamma)^{q-1} +1 -\gamma^{q-1} & (1-\gamma)^{q-1} -1  
\end{array}\right]
\end{align*}
and, as $\det Q = -\det P$, $P$ is invertible if and only if $Q$ is invertible. 
By adding the fourth column of $Q$ to the second and then deleting the third row and fourth columns,
we obtain the $3 \times 3$ matrix
\begin{align*}
R: = \left[\begin{array}{ccc}
1 &  2^{q-1}-1 & 0   \\
-\gamma^{q-1}  & (1-\gamma)^{q-1} -1 & (1-\gamma)^{q-1} -1+\gamma^{q-1}  \\
0 & 0 & -(1-\gamma)^{q-1} +1 -\gamma^{q-1}  
\end{array}\right].
\end{align*}
Since $\det R = (1-2^{q-1}) \det Q$ and $2^{q-1} \neq 1$,
$Q$ is invertible if and only if $R$ is invertible.%

We now calculate that
\begin{align*}
\det R = ((1 -\gamma^{q-1})-(1-\gamma)^{q-1})((1-\gamma)^{q-1} +(2\gamma)^{q-1} - (1 +\gamma^{q-1})),
\end{align*}
thus $R$ is not invertible if and only if either $1 -\gamma^{q-1} = (1-\gamma)^{q-1}$ or 
\begin{align}\label{eqn:k7-k3}
(2\gamma)^{q-1} - \gamma^{q-1} + (1-\gamma)^{q-1} -1 = (2^{q-1} - 1)\gamma^{q-1} + (1-\gamma)^{q-1} - 1 =0
\end{align}
By Lemma \ref{l:powerinequality},
as $q  \neq 2$ and $1 >\gamma$,
the first equality cannot hold,
thus $R$ is invertible if and only if Equation (\ref{eqn:k7-k3}) does not hold.

Consider the continuous function $f:\mathbb{R} \rightarrow \mathbb{R}$ with
\begin{align*}
f(x) := (2^{q-1} - 1)x^{q-1} + (1-x)^{q-1} - 1.
\end{align*}
Note that $f(1) = 2^{q-1}-2 \neq 0$, as $q \neq 1$,
and so we can choose $\gamma \in (0,1)$ such that $f(\gamma) \neq 0$.
Thus Equation \ref{eqn:k7-k3} does not hold and $R$ is invertible.
This now implies that $R(G,p)$ has linearly independent rows,
thus $K_7 - K_3$ is independent in $\ell_q^3$.
Since $f_3(K_7 - K_3) = 3$ also,
we have that $K_7 - K_3$ is minimally rigid in $\ell_q^3$.
\end{proof}

\begin{theorem}\label{thm:projplane}
Let $G=(V,E)$ be a triangulation of the projective plane.
Then $G$ is minimally rigid in $\ell_q^3$ for all $q\in (1,\infty)$, $q\neq 2$.
\end{theorem}

\begin{proof}
We prove the result by induction on $|V|$. Corollary \ref{c:K2d}(ii) shows that $K_6$ is minimally rigid in $\ell_q^3$ and Lemma \ref{lem:k7-k3} shows that $K_7-K_3$ is minimally rigid in $\ell_q^3$. 
Let $G=(V,E)$ be a triangulation of the projective plane. Proposition \ref{prop:projplane} shows that $G$ can be generated from $K_6$ or $K_7-K_3$ by topological vertex splitting operations. 
We can now apply Proposition \ref{p:split} to show that $G$ is minimally rigid in $\ell_q^3$ completing the proof. 
\end{proof}

\end{document}